\documentclass[pdflatex]{sn-jnl}

\usepackage{float}
\usepackage{graphicx}%
\usepackage{multirow}%
\usepackage{amsmath,amssymb,amsfonts}%
\usepackage{amsthm}%
\usepackage[scr=rsfs]{mathalpha}%
\usepackage[title]{appendix}%
\usepackage{xcolor}%
\usepackage{textcomp}%

\usepackage{manyfoot}%
\usepackage{booktabs}%
\usepackage{algorithm}%
\usepackage{algorithmicx}%
\usepackage{algpseudocode}%
\usepackage{listings}%
\usepackage{tikz}
\usetikzlibrary{arrows.meta, positioning, calc}

\newtheorem{theorem}{Theorem}[section]
\newtheorem{lemma}[theorem]{Lemma}
\newtheorem{proposition}[theorem]{Proposition}
\newtheorem{corollary}[theorem]{Corollary}

\theoremstyle{definition}
\newtheorem{definition}[theorem]{Definition}

\newtheorem{remark}[theorem]{Remark}

\begin{document}

\title[A Tractable Pseudo-Metric on Non-Parametric Exponential Statistical Manifolds via SPD Geometry]{A Tractable Pseudo-Metric on Non-Parametric Exponential Statistical Manifolds via SPD Geometry}


\author*[1]{\fnm{Amit} \sur{Vishwakarma}}
\email{amitvishwakarma.22@res.iist.ac.in}

\author[1]{\fnm{K.S. Subrahamanian} \sur{Moosath}}
\email{smoosath@iist.ac.in}

\equalcont{These authors contributed equally to this work.}

\affil[1]{\orgdiv{} 
\orgname{Indian Institute of Space Science and Technology (IIST)}, 
\orgaddress{\street{Valiamala}, \city{Thiruvananthapuram}, \postcode{695547}, \state{Kerala}, \country{India}}}


\abstract{Computing distances between probability distributions on non-parametric statistical manifolds is fundamentally intractable. The geodesic
equations live in infinite-dimensional function spaces and admit no general closed-form solution. We develop a two-stage framework that produces a computable pseudo-metric on the Pistone--Sempi exponential manifold and apply it to two-sample hypothesis testing. In the first stage, an arbitrary distribution is projected onto a chosen finite-dimensional parametric exponential family via moment-matching. This projection is many-to-one, so the resulting object is a pseudo-metric rather than a true metric. In the second stage, the parametric family is embedded into the manifold of symmetric positive definite matrices via the expected outer product of the augmented sufficient statistics vector. The embedding is a smooth diffeomorphism. The induced pullback metric differs from the Fisher--Rao metric by an explicit correction involving third-order joint cumulants of the sufficient statistics; the correction vanishes for the Gaussian family, recovering the Calvo--Oller embedding as a special case. The affine-invariant Riemannian metric on the ambient matrix manifold then provides a closed-form lower bound for the pseudo-metric, computable directly from sample moments. Applied to two-sample testing, the framework produces a test statistic
that is affine-invariant and requires no continuous tuning parameters such as a bandwidth. The choice of target exponential family determines which moments are compared. Critical values are obtained by permutation.}

\keywords{Information geometry, Symmetric positive definite matrices, Affine-invariant Riemannian metric, Non-parametric exponential manifold, Two-sample hypothesis testing}



\maketitle

\section{Introduction}\label{sec1}
The measurement of distance between probability distributions on non-parametric statistical models is a fundamental problem in statistics, information theory, and machine learning. Classical divergence measures such as the Kullback-Leibler divergence \cite{r2}, total variation distance, and Hellinger distance serve as standard theoretical tools for quantifying discrepancies between probability distributions in non-parametric settings \cite{r59}. However, as scalar measures of discrepancy, they do not provide a computable distance on the infinite-dimensional space of distributions. This leaves the construction of a tractable distance on the non-parametric manifold as a computational challenge, which this paper addresses.

The field of information geometry, pioneered by Rao \cite{r1} and systematically developed by Chentsov \cite{r4}, Amari \cite{r5}, and others, provides 
a differential geometric framework for studying families of probability 
distributions. The Fisher--Rao metric endows the space of probability densities 
with a Riemannian structure, enabling the definition of geodesic distances that 
respect the statistical properties of the underlying distributions 
\cite{r7}. For parametric exponential families, the Fisher--Rao 
metric admits an explicit form as the Hessian of the log-partition function 
and the resulting geometry exhibits a dually flat structure with exponential 
and mixture connections \cite{r5}. Despite this rigorous foundation, computing Fisher--Rao distances is hard in general, as it requires solving a boundary value problem for geodesic equations on the statistical manifold. For non-parametric manifolds, this difficulty becomes fundamentally intractable, since the geodesic equations must be solved in infinite-dimensional function spaces.

The extension of information geometry to the infinite-dimensional setting was initiated by Pistone and Sempi~\cite{r11}. They constructed a manifold structure on the space of all probability densities equivalent to a reference measure, modeling the tangent space on an Orlicz space of exponentially integrable random variables. Pistone and Rogantin 
\cite{r12} subsequently developed the theory of mean parameters and orthogonality within this framework. Subsequent contributions 
by Cena \cite{r13}, Grasselli \cite{r16}, and Santacroce 
et al.\ \cite{r17} explored dual connections, geometric 
properties, and mixture models within these infinite-dimensional manifolds. 
While this body of work establishes a rigorous framework on non-parametric space, it does not address the construction of a computable distance 
structure that inherits the manifold's Riemannian geometry within the 
Pistone--Sempi framework.

Independently, the differential geometry of symmetric positive definite (SPD) 
matrices has emerged as a rich area with applications in diffusion tensor imaging, computer vision and machine learning \cite{r18, r19, r20,r24}. The space $\mathrm{SPD}(d)$ of $d \times d$ SPD matrices admits a natural Riemannian metric known as the affine-invariant Riemannian metric (AIRM), studied extensively by Bhatia \cite{r21} and Moakher \cite{r22}. The AIRM has a closed form expression and is invariant under congruence transformations. A seminal connection between information geometry and SPD manifolds was established by Calvo and Oller~\cite{r25}, who showed that the statistical manifold of $d$-dimensional multivariate Gaussian distributions embeds isometrically into $\mathrm{SPD}(d+1)$ via a map involving the mean vector and covariance matrix and they derived a lower bound for Fisher--Rao distance between Gaussians via the closed-form AIRM formula. Nielsen \cite{r48, r47} subsequently extended this by developing pullback cone distance on the Calvo--Oller embedded submanifold and providing an approximation method for Fisher--Rao distances between Gaussians. Recently, this embedding framework was extended to the statistical manifold of Gaussian Mixture Models (GMMs) and obtained a holistic similarity measure \cite{r60}. However, all of these constructions remain structurally restricted to the Gaussian family or its finite mixtures. To the best of our knowledge, such embedding frameworks for arbitrary exponential families or non-parametric distributions remain largely unexplored.

The problem of approximating infinite-dimensional statistical models by 
finite-dimensional families has been studied via information projections. 
Csisz\'{a}r~\cite{r30} established the foundational role 
of I-divergence in projection problems, proving that the I-projection onto 
convex families of distributions is characterized by moment-matching 
constraints. In the exponential family context, this leads to the $m$-projection, characterized by the preservation of expectations of 
sufficient statistics, whose geometric interpretation via dual connections 
was developed by Amari~\cite{r5, r6}. Brigo and Pistone~\cite{r56} applied this projection philosophy to approximate solutions of infinite-dimensional evolution equations, grounding their method in the non-parametric exponential statistical manifold structure of Pistone and Sempi. More recently, Cheng and 
Tong~\cite{r55} introduced an orthogonal tangent space decomposition to extract a finite-dimensional representative of the 
Fisher--Rao metric from the infinite-dimensional manifold, establishing 
local metric information and a covariate Cram\'er--Rao bound. Their 
framework operates locally at a fixed density and produces a local finite-dimensional 
metric tensor. Building towards a global perspective, the present paper leverages both projection theory and geometric embeddings to establish a computable, global pseudo-metric on the Pistone--Sempi non-parametric exponential manifold \cite{r11}.

\subsection{The framework and its application}
\label{sec:framework}
This paper develops a framework addressing this gap through a 
two-stage construction that induces a computable distance structure 
on the non-parametric statistical manifold. In the first stage, 
arbitrary non-parametric distributions, denoted by $q \in 
\mathcal{E}_p$ (the maximal exponential model centred at a 
reference measure $p$), are projected to a $d$-dimensional 
parametric exponential family $\mathcal{N} = \{p_\theta \mid \theta 
\in \Theta\}$ via the moment-matching $m$-projection $\pi^{(m)}$.  
This projection is many-to-one; that is, multiple distributions in 
$\mathcal{E}_p$ with identical sufficient statistic expectations map 
to the same parametric distribution $p_\theta \in \mathcal{N}$.

In the second stage, the parametric family $\mathcal{N}$ is embedded into $\mathrm{SPD}(d+1)$ 
via the map $\mathcal{F}\colon \mathcal{N} \to \mathrm{SPD}(d+1)$ defined by $\mathcal{F}(p_\theta) 
= \mathbb{E}_{p_\theta}[VV^\top]$, where $V$ is the $(d+1)$-dimensional augmented 
sufficient statistics vector. This embedding $\mathcal{F}$ is a smooth 
diffeomorphism onto a submanifold $\mathcal{M} = \mathcal{F}(\mathcal{N}) 
\subset \mathrm{SPD}(d+1)$ and pulling back the AIRM via $\mathcal{F}$ 
yields the Riemannian distance $d_\mathcal{F}$ on the space $\mathcal{N}$. This induced metric differs from the 
Fisher--Rao metric through an explicit correction involving third joint cumulants of the sufficient statistics. The AIRM on the 
ambient space provides a closed-form lower bound for the intrinsic 
distance on $\mathcal{M}$, with equality on fixed-mean parameter 
slices, which are shown to be totally geodesic submanifolds of 
$\mathrm{SPD}(d+1)$.

The metric structure is then transferred back to $\mathcal{E}_p$ through 
the projection $\pi^{(m)}$. The full construction is visualized in 
Figure~\ref{fig:framework1}. The resulting pseudo-metric on 
$\mathcal{E}_p$ is
\begin{equation}
\tilde{d}(q_1, q_2) := d_{\mathcal{F}}\!\bigl(\pi^{(m)}(q_1),\, 
\pi^{(m)}(q_2)\bigr),
\end{equation}
where $d_{\mathcal{F}}$ is the Riemannian distance on the space
$\mathcal{N}$ induced by pulling back the AIRM via $\mathcal{F}$. The composition of the two stages yields a closed-form lower bound for 
$\tilde{d}$ in terms of the AIRM distance between the corresponding SPD matrix embeddings, computable directly from sample moments; the 
precise inequality and its geometric characterization are established in Section~\ref{sec:complete_framework}.

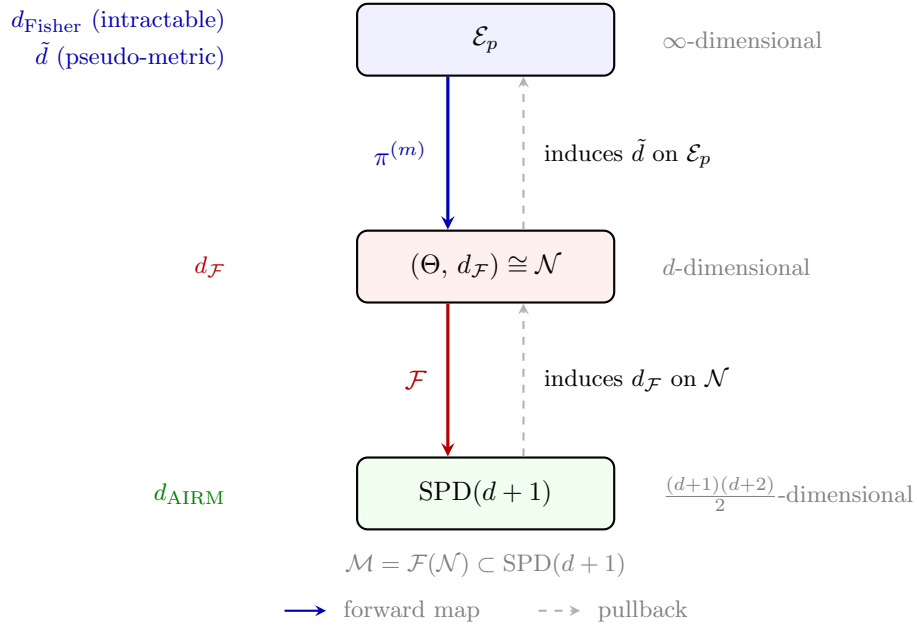
\begin{figure}[!htbp]
\centering
\begin{tikzpicture}[
    node distance=2.4cm,
    every node/.style={font=\normalsize},
    space/.style={
        draw, rounded corners=4pt, minimum width=3.4cm,
        minimum height=0.95cm, font=\normalsize, thick
    },
    map/.style={->, thick, >=stealth},
    dmap/.style={->, dashed, thick, >=stealth, gray!60}
]

\node[space, fill=blue!6]  (Ep)  at (0, 6.8) {$\mathcal{E}_p$};
\node[space, fill=red!6]   (N)   at (0, 3.8) {$(\Theta,\, d_\mathcal{F}) \cong \mathcal{N}$};
\node[space, fill=green!6] (SPD) at (0, 0.8) {$\mathrm{SPD}(d+1)$};

\node[font=\small, text=gray, right=0.5cm of Ep]
    {$\infty$-dimensional};
\node[font=\small, text=gray, right=0.5cm of N]
    {$d$-dimensional};
\node[font=\small, text=gray, right=0.5cm of SPD]
    {$\tfrac{(d+1)(d+2)}{2}$-dimensional};

\node[font=\small, left=1.6cm of Ep,  align=right, text=blue!70!black]
    {$d_{\mathrm{Fisher}}$ \scriptsize(intractable)\\[3pt]
     $\tilde{d}$ \scriptsize(pseudo-metric)};
\node[font=\small, left=1.6cm of N,   text=red!70!black]   {$d_{\mathcal{F}}$};
\node[font=\small, left=1.6cm of SPD, text=green!50!black] {$d_{\mathrm{AIRM}}$};

\draw[map, blue!70!black, line width=1.2pt]
    ([xshift=-0.5cm]Ep.south) -- ([xshift=-0.5cm]N.north)
    node[midway, left=4pt] {$\pi^{(m)}$};

\draw[map, red!70!black, line width=1.2pt]
    ([xshift=-0.5cm]N.south) -- ([xshift=-0.5cm]SPD.north)
    node[midway, left=4pt] {$\mathcal{F}$};

\draw[dmap]
    ([xshift=0.5cm]SPD.north) -- ([xshift=0.5cm]N.south)
    node[midway, right=4pt, font=\small, text=black]
        {induces $d_\mathcal{F}$ on $\mathcal{N}$};

\draw[dmap]
    ([xshift=0.5cm]N.north) -- ([xshift=0.5cm]Ep.south)
    node[midway, right=4pt, font=\small, text=black]
        {induces $\tilde{d}$ on $\mathcal{E}_p$};

\node[font=\scriptsize, text=gray, below=0.18cm of SPD]
    {$\mathcal{M} = \mathcal{F}(\mathcal{N}) \subset \mathrm{SPD}(d+1)$};

\node[font=\scriptsize, text=gray, below=0.8cm of SPD, align=center]
    (legend)
    {\tikz[baseline=-0.6ex]
         \draw[->, thick, blue!70!black, >=stealth](0,0)--(0.55,0);
     \ forward map
     \qquad
     \tikz[baseline=-0.6ex]
         \draw[->, dashed, thick, gray!60, >=stealth](0,0)--(0.55,0);
     \ pullback};

\end{tikzpicture}
\caption{The two-stage framework. Solid arrows denote forward maps:
$\pi^{(m)}$ reduces the infinite-dimensional model to the parametric
family via moment matching, and $\mathcal{F}$ embeds the parametric family
$\mathcal{N}$ into $\mathrm{SPD}(d+1)$ via the expected outer product.
Dashed arrows denote pullbacks: the AIRM on $\mathrm{SPD}(d+1)$
induces the distance $d_\mathcal{F}$ on $\mathcal{N}$, which in turn induces the
pseudo-metric $\tilde{d}$ on $\mathcal{E}_p$.}
\label{fig:framework1}
\end{figure}

Within the exponential family $\mathcal{E}_p$, when the sufficient
statistics admit a polynomial dependence structure, the full embedding
$\mathcal{F}$ can be replaced by a partial embedding $\mathcal{F}_J$ that uses only
a subset $J$ of the sufficient statistics. Specifically, if each
omitted statistic can be expressed as a polynomial of degree at most
two in the retained ones, then $\mathcal{F}_J\colon\mathcal{N}\to\mathrm{SPD}(|J|+1)$
remains injective and the same metric structure is induced on a
lower-dimensional SPD manifold. In the case of the multivariate
normal distribution, the linear sufficient statistics alone satisfy
this condition, and the partial embedding recovers the classical
Calvo--Oller embedding~\cite{r25} as a special case. For the Gamma
distribution, no such polynomial relation exists between the sufficient
statistics $\log x$ and $x$, so the full $3\times 3$ embedding is
required; in this case the pullback metric differs from the Fisher--Rao
metric by an explicit third-order cumulant correction, quantified
numerically in Section~\ref{subsec:example_gamma}.

As an application, we use the framework to construct a
geometrically principled two-sample test. Given independent samples
from unknown distributions $q_1, q_2 \in \mathcal{E}_p$, one
estimates the $m$-projection of each distribution via sample moments,
forms the corresponding SPD embedding matrices, and uses the AIRM
distance as the test statistic. The test is affine-invariant by
construction, since the AIRM is invariant under congruence
transformations.

We compare the test against two standard methods. Hotelling's $T^2$ is affine-invariant and requires no tuning, but
detects only mean differences and has no power against alternatives
that differ solely in covariance structure. Maximum mean discrepancy (MMD) with a characteristic kernel can detect arbitrary distributional differences, but it requires selecting a kernel and a bandwidth $\sigma$. The AIRM test also requires a choice, the family $\mathcal{N}$, which determines which moments are compared; this is a direct, interpretable 
decision about which features of the distributions are relevant, rather 
than a continuous numerical parameter. Once
$\mathcal{N}$ is fixed, no bandwidth or kernel enters the statistic and affine invariance holds exactly. A detailed power comparison appears in Section~\ref{subsec:hypothesis_testing}.
\subsection{Contributions}
\begin{enumerate}
\item The map $\mathcal{F}\colon \mathcal{N} \to \mathrm{SPD}(d+1)$ is a
smooth diffeomorphism onto its image, and the induced pullback metric differs from the Fisher--Rao metric by a correction term involving the third joint cumulants of the sufficient statistics.

\item The inequality $d_{\mathrm{AIRM}} \leq d_{\mathcal{F}}$ is
established, with equality when the AIRM geodesic lies entirely
within $\mathcal{M}.$

\item A polynomial criterion identifies when a partial embedding
$\mathcal{F}_J$ remains injective; the fixed-mean slices of
$\mathcal{F}_J$ are totally geodesic submanifolds of $\mathrm{SPD}(|J|+1)$, giving equality in the distance bound on these slices. For the multivariate Gaussian this recovers the
Calvo--Oller embedding~\cite{r25} as the canonical special case, while for families such as the Gamma distribution the full
embedding is required.

\item Composing $\mathcal{F}$ with the $m$-projection $\pi^{(m)}$ induces
a pseudo-metric $\tilde{d}$ on $\mathcal{E}_p$, vanishing
precisely when two distributions share the same sufficient
statistic expectations.

\item The framework gives rise to an affine-invariant two-sample test,
using the AIRM distance between empirical SPD embedding matrices as
the test statistic, with exact level $\alpha \in (0,1)$ via
permutation under exchangeability, where $\alpha$ is the prescribed
significance level.

\end{enumerate}

\subsection{Organization}
\label{subsec:organization}
 
Section~\ref{sec:related} reviews related work.
Section~\ref{sec:preliminaries} establishes mathematical
preliminaries: the Pistone--Sempi framework, Orlicz spaces, the
maximal exponential model, and SPD matrix geometry.
Section~\ref{sec:projection} develops the $m$-projection from
infinite-dimensional to finite-dimensional models and establishes
the induced pseudo-metric on $\mathcal{E}_p$.
Section~\ref{sec:spd_embedding} defines the SPD embedding,
establishes the diffeomorphism property, and derives the pullback
metric and its relation to the Fisher--Rao metric.
Section~\ref{sec:partial_embedding} develops partial embeddings
and the polynomial degree criterion for dimension reduction.
Section~\ref{sec:complete_framework} synthesizes the complete distance framework and establishes the main distance inequality. Section~\ref{sec:examples} provides worked examples for the Gaussian and Gamma families and applies the framework to geometric two-sample hypothesis testing.

\section{Related Work}\label{sec:related} The problem of defining computable distances between probability distributions 
has been approached from several distinct mathematical frameworks. We organize 
the existing literature into geometric and functional-analytic approaches on the space of distributions and identify where the 
present work departs from each line of research.

\subsection{Distances via Hilbert space embeddings}

The kernel mean embedding framework~\cite{r34}
maps each distribution $P$ to an element
$\mu_P = \mathbb{E}_{X \sim P}[k(X, \cdot)]$ of a reproducing kernel
Hilbert space associated with a positive definite
kernel $k$, and defines the MMD as
$\mathrm{MMD}(P,Q) = \|\mu_P - \mu_Q\|_{\mathcal{H}}$. With a
characteristic kernel, the MMD is a true metric on the space of
probability measures~\cite{r34}, and admits
consistent two-sample tests whose critical values are obtained via
permutation~\cite{r35}. The AIRM two-sample test
developed in Section~\ref{subsec:hypothesis_testing} adopts this same permutation mechanism to obtain critical values.

Both MMD and the present work embed distributions into a geometric
space and measure the distance between their images. However, with a
characteristic kernel, the MMD embedding targets an
infinite-dimensional Hilbert space with no Riemannian structure,
and 
requires the selection of a kernel and bandwidth parameter, both of which affect the resulting distance in ways that lack a direct geometric interpretation in terms of distribution moments. In contrast, the SPD embedding developed here targets a finite-dimensional Riemannian
manifold where the induced metric differs from the Fisher--Rao metric through an explicit third-order moment correction (Theorem~\ref{thm:induced_metric}). Ultimately, both approaches require a structural choice. MMD requires 
specifying a kernel and bandwidth, while the present framework requires 
choosing a target exponential family $\mathcal{N}$. This choice of 
$\mathcal{N}$ directly determines which moments of the distributions are 
being compared. For example, choosing $\mathcal{N}$ to be the Gaussian 
family encodes mean and variance, while adding higher-order sufficient statistics captures skewness or kurtosis.

\subsection{Non-parametric Fisher geometry via the square-root embedding}
 
The map $f \mapsto \sqrt{f}$ embeds the space of probability densities
into the positive part of the unit sphere in $L^2(\mu)$, and the
Fisher--Rao geometry on the space of densities is isometric to the
standard spherical geometry on this sphere~\cite{r51, r50}. This
equivalence provides closed-form expressions for geodesics and exponential
maps, which Srivastava et al.~\cite{r51} exploited for shape analysis and
PDF comparison in computer vision, and Holbrook et al.~\cite{r50} applied
to Bayesian nonparametric density estimation via a $\chi^2$-process prior.
 
The framework gives exact Fisher--Rao geodesics, but computing the
geodesic distance between two arbitrary densities requires evaluating the
integral $\int\!\sqrt{q_1 q_2}\,d\mu$, which has no closed form for
general non-parametric densities and presents significant numerical
challenges. The framework of Srivastava et al.~\cite{r51} operates on
densities given in explicit functional form and does not directly address
distance computation from finite samples. The present construction takes a
different approach; by projecting to a finite-dimensional exponential
family via moment matching and embedding into $\mathrm{SPD}(d+1)$, a
closed-form lower bound on the pseudo-metric is obtained directly from
sample moments, at the cost of identifying distributions with identical
low-order moments.

\subsection{Tangent space decomposition for non-parametric information
geometry}
 
Cheng and Tong~\cite{r55} recently introduced an orthogonal decomposition
of the tangent space $T_f\mathcal{M}_\mu = \mathcal{S} \oplus
\mathcal{S}^\perp$, where $\mathcal{M}_\mu$ denotes the general space of probability densities (formally defined in Section~\ref{sec:preliminaries}) and $\mathcal{S}$ is a finite-dimensional covariate
subspace, and defined the covariate Fisher information matrix (cFIM) $G_f$
as a finite-dimensional representative of the restricted metric. They
established a trace theorem $H_G(f) = \mathrm{tr}(G_f)$ providing an
information-geometric foundation for the $G$-entropy, and derived a
covariate Cram\'er--Rao lower bound connecting $G_f$ to the curvature of
the KL-divergence. While both approaches extract tractable
finite-dimensional geometry from infinite-dimensional manifolds, the cFIM
provides a local metric tensor at a fixed density suited for
semi-parametric efficiency analysis. In contrast, the present paper
constructs a global distance function on $\mathcal{E}_p$ via composed
projections, yielding a tractable global pseudo-metric rather than a local
metric tensor.

\subsection{Projection-based dimensionality reduction on statistical
manifolds}
 
Brigo and Pistone~\cite{r56} studied the projection of infinite-dimensional measure-valued evolution equations, specifically the Fokker--Planck and Kushner--Stratonovich equations, onto finite-dimensional exponential and mixture families within the Pistone--Sempi framework. They
established that the exponential manifold projection is consistent with the Fisher--Rao metric and, when the sufficient statistics are chosen as eigenfunctions of the backward diffusion operator, the projection recovers the maximum likelihood estimator for the Fokker--Planck equation. Their work applies $m$-projection to approximate solutions of partial differential equations rather than to construct a distance structure between distributions.

The approaches discussed above address distinct aspects of the problem of
measuring distances between probability distributions, but none constructs
a global pseudo-metric on the infinite-dimensional Pistone--Sempi
exponential manifold that is computable from finite sample moments and
within the geometry of $\mathcal{E}_p$ itself. Hilbert space embeddings produce computable metrics but operate outside the
exponential family geometry and require either kernel selection or
high-dimensional optimization. Closed-form Fisher--Rao constructions are
restricted to the Gaussian family. The square-root embedding gives exact
Fisher--Rao geodesics on the space of densities, but evaluating the
geodesic distance requires computing $\int\!\sqrt{q_1 q_2}\,d\mu$, which
has no closed form for general non-parametric distributions. The tangent
space decomposition provides local metric information at a fixed density
but not a global distance function. Projection-based methods operate on
the same manifold but target dynamical approximation rather than distance
computation.
 
The framework developed in this paper addresses this gap by working within
the Pistone--Sempi manifold structure~\cite{r11}. The $m$-projection
$\pi^{(m)}$ maps $\mathcal{E}_p$ to a finite-dimensional parametric family
$\mathcal{N}$, and composing with the SPD embedding $\mathcal{F}$ induces a
pseudo-metric $\tilde{d}$ on $\mathcal{E}_p$. The pullback metric $g^\mathcal{F}$
(Definition~\ref{def:pullback_metric}) is a Riemannian 
metric on $\mathcal{N}$ that differs from the Fisher--Rao
metric by an explicit third-order moment correction
(Theorem~\ref{thm:induced_metric}). It coincides with Fisher--Rao when the third joint cumulants of the sufficient statistics vanish, a condition satisfied by the Gaussian family. The AIRM distance provides a closed-form lower bound for
$\tilde{d}$ computable directly from sample moments
(Theorem~\ref{thm:distance_bound}), without requiring densities in closed
form or solving an optimization problem. The choice of $\mathcal{N}$
determines which moment features are compared; beyond this, no continuous
tuning parameter is needed.

\section{Preliminaries}
\label{sec:preliminaries}

This section establishes the mathematical framework necessary for constructing the SPD embedding of parametric exponential families 
and the induced pseudo-metric on $\mathcal{E}_p$. The exposition follows Pistone and 
Sempi~\cite{r11} and Pistone and Rogantin~\cite{r12}, 
beginning with measure-theoretic foundations, developing the theory of Orlicz 
spaces for exponential integrability, introducing the non-parametric maximal 
exponential model, and concluding with the geometric structures central to the 
analysis.

\subsection{The Underlying Probability Space}
\label{subsec:prob_space}

The construction of an infinite-dimensional manifold structure on the 
space of probability distributions requires a measure-theoretic setting 
in which all densities are mutually absolutely continuous. This is 
achieved by working relative to a fixed reference measure.

Throughout this paper, $(X, \mathcal{X}, \mu)$ denotes a fixed 
$\sigma$-finite measure space, where $X$ is the sample space, 
$\mathcal{X}$ is a $\sigma$-algebra of measurable subsets of $X$, and 
$\mu$ is a $\sigma$-finite reference measure on $(X, \mathcal{X})$.

\begin{definition}
\label{def:density_space}
The space of probability densities with respect to $\mu$ is defined as
\begin{equation}
\mathcal{M}_\mu := \left\{ p : X \to \mathbb{R} \;\middle|\; 
p \text{ is } \mathcal{X}\text{-measurable},\; p > 0\ \mu\text{-a.e.},\; 
\int_X p\, d\mu = 1 \right\}
\end{equation}
\end{definition}

Each $p \in \mathcal{M}_\mu$ defines a probability measure
$P = p \cdot \mu$ via $P(A) = \int_A p\, d\mu$ for
$A \in \mathcal{X}$. The strict positivity condition ensures
absolute continuity of all measures in $\mathcal{M}_\mu$ with
respect to $\mu$, so that for any $p, q \in \mathcal{M}_\mu$ the
log-likelihood ratio $\log(q/p)$ is well-defined $\mu$-almost
everywhere. For $p \in \mathcal{M}_\mu$ and a random variable
$u : X \to \mathbb{R}$, the expectation under $p \cdot \mu$ is
written
\begin{equation}
\mathbb{E}_p[u] := \int_X u(x)\, p(x)\, d\mu(x).
\end{equation}

\subsection{Exponential Convergence Topology}
\label{subsec:exp_convergence}

The space $\mathcal{M}_\mu$ carries the {e-convergence} topology
introduced by Pistone and Sempi~\cite{r11}. A sequence
$(p_n)_{n\in\mathbb{N}}$ in $\mathcal{M}_\mu$ is said to be
{exponentially convergent} to $p \in \mathcal{M}_\mu$ if
$p_n \to p$ in $\mu$-probability and both ratio sequences
$(p_n/p)_{n\in\mathbb{N}}$ and $(p/p_n)_{n\in\mathbb{N}}$ are
eventually bounded in $L^a(p\cdot\mu)$ for every $a > 1$, that is,
\begin{equation}
\forall\, a > 1\colon\quad
\limsup_{n\to\infty}\,\mathbb{E}_p\!\left[\left(\frac{p_n}{p}\right)^{\!a}
\right] < +\infty
\quad\text{and}\quad
\limsup_{n\to\infty}\,\mathbb{E}_p\!\left[\left(\frac{p}{p_n}\right)^{\!a}
\right] < +\infty.
\end{equation}
We refer to Pistone and Sempi~\cite{r11} for further
properties of this topology.

\subsection{Orlicz Spaces and the Cram\'er Class}
\label{subsec:orlicz_cramer}

The statistical manifold modelled on $\mathcal{M}_\mu$ requires,
at each density $p$, a Banach space that serves as the local
coordinate domain. This space is defined via the Cram\'er class,
whose integrability condition is naturally captured by Orlicz space
theory. The relevant spaces were introduced by Pistone and
Sempi~\cite{r11} following the general theory of
Krasnosel'ski\u{\i} and Ruticki\u{\i}~\cite{r14}
and Rao and Ren~\cite{r15}.

A {Young function} is a convex, even function
$\Phi : \mathbb{R} \to [0,\infty)$ that is strictly increasing
on $[0,\infty)$ satisfying $\Phi(0) = 0$
and
\begin{equation}
  \lim_{t \to +\infty} \frac{\Phi(t)}{t} = +\infty.
\end{equation}
It admits the integral representation
\begin{equation}
  \Phi(x) = \int_0^{|x|} \varphi(t)\, dt,
\end{equation}
where $\varphi : [0,\infty) \to [0,\infty)$ is non-decreasing and
left-continuous with $\varphi(0) = 0$. The {conjugate Young function} is $\Psi(y) = \int_0^{|y|} \chi(t)\, dt$, where $\chi$ is the generalized inverse of $\varphi$. The {Young inequality} $|xy| \leq \Phi(x) + \Psi(y)$ holds for all $x, y \in \mathbb{R}$, with equality if and only if $y = \varphi(|x|)$ or $x = \chi(|y|)$.

For the measure space $(X, \mathcal{X}, \mu)$, the {Orlicz space}
$L^\Phi(\mu)$ is defined as the collection of all $\mu$-equivalence
classes of measurable functions $u : X \to \mathbb{R}$ for which there
exists $\alpha > 0$ such that $E_\mu[\Phi(\alpha u)] < +\infty$,
that is, 
\begin{equation}
  L^\Phi(\mu) := \bigl\{ u \text{ measurable} : \exists\,\alpha > 0,\;
  \textstyle\int_X \Phi(\alpha u)\, d\mu < +\infty \bigr\}.
\end{equation}
It is a Banach space under the {Luxemburg norm}
\begin{equation}
  \|u\|_\Phi := \inf\!\left\{ k > 0 :
  \int_X \Phi\!\left(\frac{u}{k}\right) d\mu \leq 1 \right\}.
\end{equation}

Three Young functions play central roles in the non-parametric geometry
of exponential families~\cite{r12}:
\begin{align}
  \Phi_1(x) &:= \cosh(|x|) - 1, \label{eq:Phi1} \\
  \Phi_2(x) &:= e^{|x|} - |x| - 1, \label{eq:Phi2} \\
  \Phi_3(x) &:= (1 + |x|)\log(1 + |x|) - |x|. \label{eq:Phi3}
\end{align}
The functions $\Phi_1$ and $\Phi_2$ are equivalent, meaning the norms
they induce are equivalent and the corresponding Banach spaces
coincide; the functions $\Phi_2$ and $\Phi_3$ are
conjugate~\cite{r12}. The function $\Phi_1$
governs the integrability condition for the Cram\'er class and
defines the local model space of the statistical manifold.

For each density $p \in \mathcal{M}_\mu$, the moment generating
function of a measurable function $u$ on $(X,\mathcal{X})$ with
respect to the probability measure $p \cdot \mu$ is
\begin{equation}
  \hat{u}_p(t) := \mathbb{E}_p\!\left[e^{tu}\right]
  = \int_X e^{tu}\, p\, d\mu, \qquad t \in \mathbb{R}.
\end{equation}

\begin{definition}[Cram\'er Class]
\label{def:cramer_class}
For $p \in \mathcal{M}_\mu$, the {Cram\'er class at $p$} is
the set of all measurable functions $u$ on $(X, \mathcal{X})$ whose
moment generating function $\hat{u}_p$ is finite in a neighbourhood
of the origin, i.e.\
\begin{equation}
  0 \in \operatorname{dom}(\hat{u}_p)^{\circ}.
\end{equation}
The {centred Cram\'er class at $p$} additionally requires
$\mathbb{E}_p[u] = 0$.
\end{definition}

\begin{proposition}[Cram\'er Class as Orlicz Space]
\label{prop:cramer_orlicz}
For each $p \in \mathcal{M}_\mu$, the Cram\'er class at $p$ coincides
with the Orlicz space $L^{\Phi_1}(p \cdot \mu)$. In particular it is a
Banach space, and the centred Cram\'er class is a closed subspace.
\end{proposition}

\begin{proof}
  See Pistone and Sempi~\cite{r11}. The Banach space structure and the norm
  formula are given explicitly in Pistone and
  Rogantin~\cite{r12}.
\end{proof}

\begin{proposition}
\label{prop:cramer_norm}
The centred Cram\'er class, denoted
\begin{equation}
  B_p := \left\{ u \in L^{\Phi_1}(p \cdot \mu) :
  \mathbb{E}_p[u] = 0 \right\},
  \label{eq:Bp_def}
\end{equation}
is a Banach space under the Luxemburg norm
\begin{equation}
  \|u\|_p := \inf\!\left\{ r > 0 :
  \mathbb{E}_p\!\left[\cosh\!\left(\frac{u}{r}\right) - 1\right]
  \leq 1 \right\},
  \label{eq:lux_norm}
\end{equation}
and $\|u\|_p \leq 1$ if and only if
 $\mathbb{E}_p\!\left[\cosh(u)\right] \leq 2$.
\end{proposition}
Since every $u \in B_p$ satisfies $\mathbb{E}_p[u] = 0$, the
bilinear form
\begin{equation}
  \langle u, v \rangle_p
  := \mathbb{E}_p[uv]
  = \mathrm{Cov}_p(u, v),
  \qquad u, v \in B_p,
  \label{eq:inner_product_bp}
\end{equation}
is well defined on $B_p$ and defines a continuous scalar product
(Pistone and Rogantin~\cite{r12},
Definition~10). The topology induced by the Luxemburg norm
$\|\cdot\|_p$ on $B_p$ is strictly stronger than the $L^2(p\cdot\mu)$
topology induced by $\langle\cdot,\cdot\rangle_p$; the latter will be
identified in Subsection~\ref{subsec:tangent_fisher} as the Fisher
information metric on the manifold.

\subsection{The Statistical Manifold Structure}
\label{subsec:manifold_structure}

The centred Cram\'er class $B_p$ serves as the local model space
for the manifold structure on $\mathcal{M}_\mu$. The key property
enabling this is that the open unit ball of $B_p$,
\begin{equation}
  \mathcal{W}_p := \{ u \in B_p : \|u\|_p < 1 \},
  \label{eq:Wp_def}
\end{equation}
is contained in the interior of the effective domain of the cumulant
generating functional introduced in Definition~\ref{def:cumulant_functional}
below, which in turn ensures that the exponential patch map
$e_p$ of Definition~\ref{def:exponential_map_chart} is well defined on $\mathcal{W}_p$.

\begin{definition}[Cumulant Generating Functional]
\label{def:cumulant_functional}
For $p \in \mathcal{M}_\mu$, the {cumulant generating functional}
is the map $K_p : B_p \to [0, +\infty]$ defined by
\begin{equation}
  K_p(u) := \log \mathbb{E}_p\!\left[e^u\right]
  = \log\!\int_X e^u\, p\, d\mu.
  \label{eq:Kp_def}
\end{equation}
Its effective domain is
\begin{equation}
  \mathrm{Dom}(K_p) := \bigl\{ u \in B_p : K_p(u) < +\infty \bigr\}.
\end{equation}
The open unit ball $\mathcal{W}_p$ satisfies
$\mathcal{W}_p \subset \mathrm{Dom}(K_p)^{\circ}$
(Pistone and Rogantin~\cite{r12}).
\end{definition}

\begin{definition}[Exponential Map and Chart]
\label{def:exponential_map_chart}
For each $p \in \mathcal{M}_\mu$, the exponential map is given by
\begin{equation}
e_p : \mathcal{W}_p \to \mathcal{M}_\mu, \qquad 
u \mapsto e^{u - K_p(u)} \cdot p.
\end{equation}
The map $e_p$ is injective; its range $U_p$ has the inverse
\begin{equation}
s_p : U_p \to \mathcal{W}_p, \qquad 
q \mapsto \log(q/p) - \mathbb{E}_p[\log(q/p)].
\end{equation}
Each pair $(U_p, s_p)$ is a chart, with $s_p$ the centred log-likelihood coordinate. The log-likelihood ratio $\log(q/p)$ is well-defined $\mu$-almost everywhere by the absolute continuity established in Subsection~\ref{subsec:prob_space}.
\end{definition}
For $p_1, p_2 \in \mathcal{M}_\mu$ with 
$U_{p_1} \cap U_{p_2} \neq \emptyset$, the transition map is 
given by
\begin{equation}
(s_{p_2} \circ e_{p_1})(u) = u + \log(p_1/p_2) - 
\mathbb{E}_{p_2}\!\left[\log(p_1/p_2) + u\right],
\label{eq:transition_map}
\end{equation}
whose derivative $B_{p_1} \ni u \mapsto u - \mathbb{E}_{p_2}[u] 
\in B_{p_2}$ is a topological linear isomorphism~\cite{r58}. 

\begin{theorem}[Affine Atlas; Pistone and 
Sempi~\cite{r11}]
\label{thm:affine_atlas}
The collection $\{(U_p, s_p) : p \in \mathcal{M}_\mu\}$ is an 
affine $C^\infty$-atlas on $\mathcal{M}_\mu$. The induced 
topology on sequences coincides with the e-convergence topology 
of Subsection~\ref{subsec:exp_convergence}, and transition maps 
are given by~\eqref{eq:transition_map}.
\end{theorem}

\begin{definition}[Exponential Statistical Manifold]
\label{def:exp_stat_manifold}
The {exponential statistical manifold} is the manifold 
defined by the atlas of Theorem~\ref{thm:affine_atlas} on 
$\mathcal{M}_\mu$.
\end{definition}

\subsection{The Maximal Exponential Model}
\label{subsec:maximal_model}

The exponential statistical manifold $\mathcal{M}_\mu$ is not 
connected. Its connected components are disjoint open subsets, 
each carrying a self-contained manifold structure. The natural 
domain for the present framework is the connected component 
containing a fixed reference density $p$, called the maximal 
exponential model.

\begin{definition}[Maximal Exponential Model]
\label{def:maximal_exponential}
For $p \in \mathcal{M}_\mu$, the {maximal exponential model at $p$}
is
\begin{equation}
  \mathcal{E}_p := \bigl\{ e^{u - K_p(u)} \cdot p :
  u \in \mathrm{Dom}(K_p)^{\circ},\;
  \mathbb{E}_p[u] = 0 \bigr\}.
  \label{eq:Ep_def}
\end{equation}
The centering condition $\mathbb{E}_p[u] = 0$ is imposed without 
loss of generality. Setting $\tilde{u} := u - \mathbb{E}_p[u]$, 
one computes $K_p(\tilde{u}) = K_p(u) - \mathbb{E}_p[u]$, 
so that $\tilde{u} - K_p(\tilde{u}) = u - K_p(u)$ and the 
density $e^{u - K_p(u)} \cdot p$ is unchanged.
\end{definition}

\begin{theorem}[Connected Component; Pistone and
Sempi~\cite{r11}, Pistone and
Rogantin~\cite{r12}]
\label{thm:connected_component}
$\mathcal{E}_p$ is the connected component of $\mathcal{M}_\mu$
containing $p$.
\end{theorem}

Within $\mathcal{E}_p$, finite-dimensional parametric exponential
families arise as submodels generated by a finite collection of
sufficient statistics. Given $u_1, \ldots, u_d \in B_p$, the
associated parametric family is
\begin{equation}
  \mathcal{N} = \left\{ e^{\,\sum_{j=1}^d \theta_j u_j - \psi(\theta)}
  \cdot p : \theta \in \Theta \right\},
  \qquad
  \psi(\theta) = K_p\!\left(\sum_{j=1}^d \theta_j u_j\right),
  \label{eq:param_model}
\end{equation}
where
\begin{equation}
  \Theta = \left\{ \theta \in \mathbb{R}^d :
  \sum_{j=1}^d \theta_j u_j \in \mathrm{Dom}(K_p)^{\circ} \right\}
\end{equation}
is an open subset of $\mathbb{R}^d$, by continuity of
$\theta \mapsto \sum_j \theta_j u_j$ and openness of
$\mathrm{Dom}(K_p)^{\circ}$. The {mean parameters} are obtained
by differentiating $\psi$ under the integral sign
(Pistone and Rogantin~\cite{r12})
\begin{equation}
  \eta_j(\theta) := \frac{\partial \psi}{\partial \theta_j}(\theta)
  = \mathbb{E}_{p_\theta}[u_j],
  \qquad j = 1, \ldots, d.
  \label{eq:mean_params}
\end{equation}
In vector form, $\eta(\theta) = \nabla\psi(\theta) \in \mathbb{R}^d$.
The vector $\eta(\theta) = (\eta_1(\theta),\ldots,\eta_d(\theta))^\top$
is called the {mixture coordinates}~\cite{r5} or the
{mean parameters}~\cite{r10}. These will
serve as the primary coordinates for the $m$-projection developed in Section~\ref{sec:projection} and the 
embedding $\mathcal{F}$ developed in Section~\ref{sec:spd_embedding}.

The suitability of $\mathcal{E}_p$ as the domain for the present
framework follows from three properties. First, it is the connected
component of $\mathcal{M}_\mu$ containing $p$ (Theorem~\ref{thm:connected_component}),
and the collection of charts $\{(U_p, s_p) : p \in \mathcal{M}_\mu\}$
forms an {affine} $C^\infty$-atlas on $\mathcal{M}_\mu$
(Pistone and Sempi~\cite{r11},
Pistone and Rogantin~\cite{r12}),
meaning the transition maps
\begin{equation}
  s_{p_2} \circ e_{p_1}(u) = u + \log\!\frac{p_1}{p_2}
  - \mathbb{E}_{p_2}\!\left[u + \log\frac{p_1}{p_2}\right]
  \label{eq:transition_map_Ep} 
\end{equation}
are affine functions of $u$. Second, for every $q \in \mathcal{E}_p$,
the centred log-likelihood $s_p(q) = \log(q/p) - \mathbb{E}_p[\log(q/p)]$
belongs to $B_p$. Third, any finite-dimensional parametric exponential
family generated from $p$ embeds naturally as a submanifold of
$\mathcal{E}_p$ (Subsection~\ref{subsubsec:parametric_submanifolds}).

\subsection{Tangent Spaces and the Fisher Information Metric}
\label{subsec:tangent_fisher}

The manifold structure established in Subsection~\ref{subsec:manifold_structure}
determines the tangent space at each density and the canonical
Riemannian metric on it. Since $\mathcal{M}_\mu$ is modelled on the
Banach spaces $B_p$, the tangent space at $p$ is identified with $B_p$
itself, and the natural inner product on $B_p$ introduced in
equation~\eqref{eq:inner_product_bp} becomes the Fisher--Rao metric.

\begin{proposition}[Tangent Space; Pistone and
Rogantin~\cite{r12}]
\label{prop:tangent_space}
Let $p(t)$ be a regular curve in $\mathcal{M}_\mu$ with $p(t_0) = p$,
and let $u(t) \in B_q$ be its representation by the chart $s_q$
centred at $q \in \mathcal{M}_\mu$, where
$t \in \{t : u(t) \in \mathrm{Dom}(K_q)^{\circ}\}$.
The tangent vector $\dot{u}(t_0)$ relates to the score function by
\begin{equation}
  \dot{u}(t_0) - \mathbb{E}_q[\dot{u}(t_0)]
  = \frac{d}{dt}\log\!\left(\frac{p(t)}{p}\right)\bigg|_{t=t_0}.
  \label{eq:tangent_score}
\end{equation}
In particular, when the chart is centred at $p$ itself
(i.e.\ $q = p$), every $u(t) \in B_p$ satisfies
$\mathbb{E}_p[\dot{u}(t_0)] = 0$, so~\eqref{eq:tangent_score}
simplifies to
\begin{equation}
  \dot{u}(t_0) = \frac{d}{dt}\log\!\left(\frac{p(t)}{p}
  \right)\bigg|_{t=t_0},
  \label{eq:tangent_score_centered}
\end{equation}
and the tangent space $T_p\mathcal{M}_\mu$ is identified with $B_p$.
\end{proposition}

\begin{definition}[Fisher Information Metric]
\label{def:fisher_metric}
For $p \in \mathcal{M}_\mu$ and $u, v \in B_p$, the {Fisher--Rao
metric} is the inner product
\begin{equation}
  I_p(u, v) := \langle u, v \rangle_p = \mathbb{E}_p[uv]
  = \mathrm{Cov}_p(u, v),
  \label{eq:fisher_metric}
\end{equation}
which coincides with the scalar product on $B_p$ introduced in
equation~\eqref{eq:inner_product_bp}. When restricted to the parametric family $\mathcal{N}$ with sufficient statistics $(u_1, \ldots, u_d)$, this metric yields the {Fisher information matrix} at $\theta \in \Theta$
\begin{equation}
  [I(\theta)]_{jk} := \mathrm{Cov}_{p_\theta}(u_j, u_k)
  = \frac{\partial^2 \psi}{\partial \theta_j \partial \theta_k}(\theta),
  \qquad j, k = 1, \ldots, d,
  \label{eq:fisher_matrix}
\end{equation}
where the second equality follows from Pistone and Rogantin~\cite{r12}, Proposition~16(c). The matrix
$I(\theta)$ is symmetric positive definite whenever $u_1, \ldots, u_d$
are linearly independent in $L^2(p \cdot \mu)$ which is assumed throughout.
\end{definition}

\begin{remark}[Second Moment Decomposition]
\label{rem:second_moment}
Define the {second moment matrix} of the sufficient statistics at
$\theta$ by
\begin{equation}
  [A(\theta)]_{jk} := \mathbb{E}_{p_\theta}[u_j u_k],
  \qquad j, k = 1, \ldots, d.
  \label{eq:A_def}
\end{equation}
Since $u_j \in B_p$ implies $\mathbb{E}_p[u_j] = 0$ but
$\mathbb{E}_{p_\theta}[u_j] = \eta_j(\theta)$ for general $\theta$
(equation~\eqref{eq:mean_params}), the standard covariance identity
$\mathrm{Cov}_{p_\theta}(u_j, u_k)
= \mathbb{E}_{p_\theta}[u_j u_k]
- \mathbb{E}_{p_\theta}[u_j]\,\mathbb{E}_{p_\theta}[u_k]$
gives
\begin{equation}
  A(\theta) = I(\theta) + \eta(\theta)\,\eta(\theta)^{\top}.
  \label{eq:A_I_decomposition}
\end{equation}
This decomposition is central to the embedding $\mathcal{F}$ developed in Section~\ref{sec:spd_embedding}.
\end{remark}

\subsection{Submanifolds and Splitting}
\label{subsec:submanifolds}
The maximal exponential model $\mathcal{E}_p$ is 
infinite-dimensional, but the framework developed in this paper 
operates through finite-dimensional parametric submodels 
$\mathcal{N} \subset \mathcal{E}_p$. This subsection establishes 
that such submodels are smooth submanifolds of 
$\mathcal{M}_\mu$, following the splitting construction of Pistone 
and Rogantin~\cite{r12}. This 
submanifold structure is the geometric prerequisite for the $m$-projection developed in Section~\ref{sec:projection}.

\subsubsection{Splitting of the Tangent Space}
\label{subsubsec:splitting}
Let $V^1_p$ and $V^2_p$ be closed subspaces of $B_p$ with
$V^1_p \cap V^2_p = \{0\}$ and $B_p = V^1_p + V^2_p$. Any element
$u \in B_p$ can then be written uniquely as $u = u_1 + u_2$ with
$u_i \in V^i_p$, defining a splitting $B_p = V^1_p \oplus V^2_p$.
The associated linear projectors $P_i : B_p \to V^i_p$, defined by
$P_i(u) = u_i$, are continuous by the Banach closed-graph
theorem~\cite{r58}.

\subsubsection{Parametric Exponential Families as Submanifolds}
\label{subsubsec:parametric_submanifolds}
Let $u_1, \ldots, u_d \in B_p$ and set $V^1_p =
\mathrm{span}\{u_1, \ldots, u_d\}$. Since $B_p \subset
L^2_0(p\cdot\mu)$ (the space of $p\cdot\mu$-square-integrable
centred functions), the scalar product $\langle\cdot,\cdot\rangle_p$
makes $B_p$ a pre-Hilbert space. The orthogonal projector
$P_{u_1,\ldots,u_d}$ from $B_p$ onto the finite-dimensional subspace
$V^1_p \subset L^2_0(p\cdot\mu)$ is characterized by the normal
equations
\begin{equation}
  \langle u, u_j \rangle_p
  = \sum_{i=1}^{d} \hat{\theta}_i\, \langle u_i, u_j \rangle_p,
  \qquad j = 1, \ldots, d,
  \label{eq:normal_equations}
\end{equation}
where $\hat{\theta} = (\hat{\theta}_1, \ldots, \hat{\theta}_d) \in
\Theta$ gives the unique coordinate representation
$P_{u_1,\ldots,u_d}(u) = \sum_i \hat{\theta}_i u_i$ of the
projection. Since $V^1_p$ is finite-dimensional, it is closed in $B_p$, and $V^2_p = \ker(P_{u_1,\ldots,u_d})$ is a closed subspace of $B_p$,
giving the splitting $B_p = V^1_p \oplus V^2_p$.
 
The parametric family $\mathcal{N}$ defined in~\eqref{eq:param_model}
is then a $d$-dimensional submanifold of $\mathcal{M}_\mu$ in the
sense of Pistone and Rogantin~\cite{r12}. The natural parameter map
\begin{equation}
  \Theta \ni \theta \;\longmapsto\;
  \sum_{j=1}^d \theta_j u_j \in V^1_p \subset B_p
  \label{eq:param_map}
\end{equation}
is a linear isomorphism between $\Theta \subset \mathbb{R}^d$ and
$V^1_p$, and the composition $p_\theta \mapsto \theta$ is a $C^\infty$
chart of $\mathcal{N}$ into $\Theta$
(Pistone and Rogantin~\cite{r12}).
This submanifold structure provides the geometric foundation for the
$m$-projection developed in Section~\ref{sec:projection} and the 
embedding $\mathcal{F}$ developed in Section~\ref{sec:spd_embedding}.

\subsection{Symmetric Positive Definite Matrices}
\label{subsec:spd}

The natural embedding target for the parametric submanifold
$\mathcal{N} \subset \mathcal{E}_p$ is the manifold of SPD matrices.
The embedding $\mathcal{F}$ (constructed in
Section~\ref{sec:spd_embedding}) sends each $p_\theta \in \mathcal{N}$
to the expected outer product $\mathbb{E}_{p_\theta}[VV^\top]$ of the
augmented sufficient statistics vector $V$.

\begin{definition}[SPD Matrices]
\label{def:spd}
The space of $m \times m$ symmetric positive definite matrices is
\[
  \mathrm{SPD}(m) := \left\{ S \in \mathbb{R}^{m \times m} :
  S = S^\top,\;
  v^\top S v > 0 \;\text{for all }\, v \in \mathbb{R}^m \setminus \{0\}
  \right\}.
\]
\end{definition}

The space $\mathrm{SPD}(m)$ is an open convex cone in the space of
symmetric matrices, of dimension $\frac{m(m+1)}{2}$.

\begin{definition}[Affine-Invariant Riemannian Metric]
\label{def:airm}
The AIRM on
$\mathrm{SPD}(m)$ is defined at $S \in \mathrm{SPD}(m)$ by
\begin{equation}
  \langle \Xi, \Omega \rangle_S^{\mathrm{AIRM}} :=
  \frac{1}{2}\,\mathrm{tr}(S^{-1} \Xi S^{-1} \Omega),
  \label{eq:airm}
\end{equation}
for tangent vectors $\Xi, \Omega \in T_S\mathrm{SPD}(m) \cong
\mathrm{Sym}(m)$. The corresponding line element is
\[
  ds^2_{\mathrm{AIRM}} =
  \frac{1}{2}\operatorname{tr}\!\bigl((S^{-1}\,dS)^2\bigr).
\]
\end{definition}

The AIRM is invariant under congruence transformations. For any
invertible matrix $A$,
\begin{equation}
  \langle A\Xi A^\top, A\Omega A^\top \rangle_{ASA^\top}^{\mathrm{AIRM}}
  = \langle \Xi, \Omega \rangle_S^{\mathrm{AIRM}}.
  \label{eq:airm_congruence}
\end{equation}
This invariance property underlies the affine invariance of the test statistic developed in 
Subsection~\ref{subsec:hypothesis_testing}.

The geodesic connecting $S_0, S_1 \in \mathrm{SPD}(m)$ is
\begin{equation}
  \gamma(t) = S_0^{1/2}
  \bigl(S_0^{-1/2} S_1 S_0^{-1/2}\bigr)^t S_0^{1/2},
  \qquad t \in [0,1],
  \label{eq:spd_geodesic}
\end{equation}
and the Riemannian distance is
\begin{equation}
  d_{\mathrm{AIRM}}(S_0, S_1) =
  \frac{1}{\sqrt{2}}\,
  \bigl\|\log\bigl(S_0^{-1/2} S_1 S_0^{-1/2}\bigr)\bigr\|_F.
  \label{eq:airm_distance}
\end{equation}
The manifold $(\mathrm{SPD}(m), \mathrm{AIRM})$ is a Hadamard manifold
(complete, simply connected, non-positive sectional curvature).
Consequently, the geodesic is unique, the distance is globally
well-defined, and the exponential map is a diffeomorphism. The
closed-form expression~\eqref{eq:airm_distance} provides a
computationally tractable lower bound for the pseudo-metric on
$\mathcal{E}_p$ developed in Section~\ref{sec:complete_framework}.


\section{The \texorpdfstring{$m$}{m}-Projection}
\label{sec:projection}

The maximal exponential model $\mathcal{E}_p$ and its
finite-dimensional parametric submodel $\mathcal{N} \subset
\mathcal{E}_p$ were established in
Subsections~\ref{subsec:maximal_model} and~\ref{subsec:submanifolds}
respectively. The present section addresses the systematic
approximation of an arbitrary distribution $q \in \mathcal{E}_p$ by
an element of the finite-dimensional family $\mathcal{N}$. Such
approximation is essential because infinite-dimensional models are
computationally intractable for distance calculations and the embedding $\mathcal{F}$ constructed in Section~\ref{sec:spd_embedding} operates on finite-dimensional sufficient statistics.

The projection is defined by minimizing the Kullback--Leibler
divergence from $q$ to $\mathcal{N}$. The central result of this
section (Theorem~\ref{thm:moment_matching}) is that this optimal
projection is characterized entirely by the matching of sufficient
statistic expectations, a property that connects the
infinite-dimensional geometry of $\mathcal{E}_p$ to the
finite-dimensional parametric structure and forms the computational bridge to the embedding $\mathcal{F}$. The many-to-one nature of this
projection is then shown to induce a pseudo-metric on $\mathcal{E}_p$
rather than a true metric, whose computation is made tractable in
Section~\ref{sec:complete_framework}.

\subsection{The Kullback--Leibler Divergence}
\label{subsec:kl_divergence}
The projection onto $\mathcal{N}$ is defined through 
minimization of the Kullback--Leibler divergence. For probability densities $q, r \in \mathcal{M}_\mu$, the 
{KL divergence from $q$ to $r$} is
\begin{equation}
D_{\mathrm{KL}}(q \| r) := \int_X q(x) 
\log\frac{q(x)}{r(x)} \, d\mu(x) 
= \mathbb{E}_q\!\left[\log(q/r)\right].
\end{equation}
The KL divergence is non-negative, with 
$D_{\mathrm{KL}}(q \| r) = 0$ if and only if $q = r$, $\mu$-almost everywhere~\cite{r3}. It is not 
symmetric in general and the two orderings of its 
arguments give rise to two distinct projections onto 
$\mathcal{N}$, defined in the next subsection.

\subsection{Information Projection}
\label{subsec:m_projection}
Given $q \in \mathcal{E}_p$ and the parametric submodel 
$\mathcal{N} \subset \mathcal{E}_p$, two projections arise 
from the asymmetry of $D_{\mathrm{KL}}$.

\begin{enumerate}[label=(\alph*)]
    \item The {$m$-projection} of $q$ onto $\mathcal{N}$ is
    \begin{equation}
    \pi^{(m)}(q) := \arg\min_{p_\theta \in \mathcal{N}} 
    D_{\mathrm{KL}}(q \| p_\theta).
    \end{equation}
    \item The {e-projection} of $q$ onto $\mathcal{N}$ is
    \begin{equation}
    \pi^{(e)}(q) := \arg\min_{p_\theta \in \mathcal{N}} 
    D_{\mathrm{KL}}(p_\theta \| q).
    \end{equation}
\end{enumerate}

The $m$-projection is the one relevant to the present framework 
because its optimality condition reduces to moment matching 
(Theorem~\ref{thm:moment_matching}), making it computable from 
sample means. The $e$-projection, which arises from the reverse 
KL ordering, does not admit such a characterization in general 
and is not considered further.

\subsection{Moment Matching Characterization}
\label{subsec:moment_matching}

The key structural property of the $m$-projection is that 
it is fully determined by the mean parameters of $q$ with 
respect to the sufficient statistics of $\mathcal{N}$. 
The following theorem characterizes the $m$-projection 
through moment matching and establishes its uniqueness.
The existence of the $m$-projection under a natural condition on $\eta^q$ is 
established in Theorem~\ref{thm:projection_existence}.

\begin{theorem}[Moment Matching Characterization]
\label{thm:moment_matching}
Let $q \in \mathcal{E}_p$ and let $\mathcal{N} = 
\{p_\theta : \theta \in \Theta\}$ be the parametric 
exponential model with sufficient statistics 
$(u_1, \ldots, u_d)$ and log-partition function $\psi$ 
as in~\eqref{eq:param_model}. If the $m$-projection 
$\pi^{(m)}(q) = p_{\theta^*}$ exists, it is unique and 
is characterized by the {moment matching equations}
\begin{equation}
\label{eq:moment_matching}
\mathbb{E}_{p_{\theta^*}}[u_j] = \mathbb{E}_q[u_j], 
\quad j = 1, \ldots, d,
\end{equation}
or equivalently, $\theta^*$ is the unique solution to
\begin{equation}
\label{eq:gradient_matching}
\frac{\partial \psi}{\partial \theta_j}(\theta^*) = 
\mathbb{E}_q[u_j], \quad j = 1, \ldots, d.
\end{equation}
\end{theorem}

\begin{proof}
This result is classical for exponential 
families~\cite{r30, r7}; 
the proof below establishes it within the present 
framework using the notation of 
Section~\ref{sec:preliminaries}.

Expanding the KL divergence
\begin{align}
D_{\mathrm{KL}}(q \| p_\theta) 
= \int_X q \log q \, d\mu 
- \int_X q \log p_\theta \, d\mu.
\end{align}
The first term is independent of $\theta$, so minimizing 
$D_{\mathrm{KL}}(q \| p_\theta)$ over $\Theta$ is 
equivalent to maximizing $L(\theta) := 
\mathbb{E}_q[\log p_\theta]$. Substituting 
$\log p_\theta(x) = \sum_{j=1}^d \theta_j u_j(x) - 
\psi(\theta) + \log p(x)$ from~\eqref{eq:param_model} 
and using linearity of expectation
\begin{equation}
L(\theta) = \sum_{j=1}^d \theta_j \mathbb{E}_q[u_j] 
- \psi(\theta) + \mathbb{E}_q[\log p],
\end{equation}
where $\mathbb{E}_q[u_j]$ are constants fixed by $q$ 
and $\mathbb{E}_q[\log p]$ is constant in $\theta$. 
Differentiating with respect to $\theta_k$ and setting 
to zero
\begin{equation}
\frac{\partial L}{\partial \theta_k}(\theta) = 
\mathbb{E}_q[u_k] - \frac{\partial \psi}{\partial 
\theta_k}(\theta) = 0 \quad \Longrightarrow \quad 
\frac{\partial \psi}{\partial \theta_k}(\theta^*) = 
\mathbb{E}_q[u_k].
\end{equation}
By~\eqref{eq:mean_params}, $\frac{\partial \psi}
{\partial \theta_k}(\theta) = \mathbb{E}_{p_\theta}[u_k]$ 
for all $\theta \in \Theta$, so at the critical 
point $\theta^*$
\begin{equation}
\mathbb{E}_{p_{\theta^*}}[u_k] = 
\frac{\partial \psi}{\partial \theta_k}(\theta^*) = 
\mathbb{E}_q[u_k], \quad k = 1, \ldots, d,
\end{equation}
establishing~\eqref{eq:moment_matching} 
and~\eqref{eq:gradient_matching}. For uniqueness, 
the Hessian of $L$ is
\begin{equation}
\nabla^2 L(\theta) = 
-\nabla^2\psi(\theta) = -I(\theta),
\end{equation}
where $I(\theta)$ is the Fisher information 
matrix~\eqref{eq:fisher_matrix}. Since $I(\theta) 
\succ 0$ for all $\theta \in \Theta$ by the linear 
independence of $u_1, \ldots, u_d$ in $B_p$ 
(Subsection~\ref{subsec:tangent_fisher}), the Hessian 
of $L$ is negative definite everywhere, making $L$ 
strictly concave. A strictly concave function on an 
open convex domain has at most one critical point, and 
any critical point is a global maximum. The critical 
point $\theta^*$ is therefore the unique global 
maximizer of $L(\theta)$, and hence the unique 
minimizer of $D_{\mathrm{KL}}(q \| p_\theta)$ 
over $\mathcal{N}$. Existence of such a critical 
point equivalently, solvability of the moment 
matching equations requires $\eta^q$ to lie in 
the image of $\nabla\psi$, and is established 
separately in Theorem~\ref{thm:projection_existence}.

\end{proof}

\begin{remark}
\label{rem:moment_geometric}
Theorem~\ref{thm:moment_matching} shows that the
$m$-projection depends on $q$ only through its moment
vector $\eta^q := (\mathbb{E}_q[u_1], \ldots,
\mathbb{E}_q[u_d])^\top \in \mathbb{R}^d$. Two
distributions $q, q' \in \mathcal{E}_p$ with
$\eta^q = \eta^{q'}$ have the same $m$-projection,
regardless of how they differ in higher-order
structure. This is the direct cause of the many-to-one
behaviour established in
Subsection~\ref{subsec:many_to_one}, and $\eta^q$
itself is the mean parameter vector introduced in
Subsection~\ref{subsec:maximal_model}.
\end{remark}
The geometric relationship between $q$, $\mathcal{N}$, 
and $\pi^{(m)}(q)$ is illustrated in 
Figure~\ref{fig:m_projection_geometry}.
\subsection{Existence of \texorpdfstring{$m$}{m}-Projection}
\label{subsec:projection_existence}

Theorem~\ref{thm:moment_matching} establishes that 
any $m$-projection, if it exists, is necessarily unique 
and is characterized by the moment matching 
equations~\eqref{eq:moment_matching}. The remaining 
question is when such a $\theta^*$ actually exists, 
that is, when the moment vector $\eta^q$ of the target 
distribution $q$ is achievable by some element of 
$\mathcal{N}$.

\begin{theorem}[Existence of $m$-Projection]
\label{thm:projection_existence}
Let $q \in \mathcal{E}_p$ and let
\begin{equation}
\mathcal{H}_{\eta} := \nabla\psi(\Theta) = 
\bigl\{ \eta(\theta) \in \mathbb{R}^d : 
\theta \in \Theta \bigr\}
\end{equation}
be the interior of the mean parameter space of $\mathcal{N}$. If 
$\eta^q = (\mathbb{E}_q[u_1], \ldots, 
\mathbb{E}_q[u_d])^\top \in \mathcal{H}_{\eta}$, then the 
$m$-projection $\pi^{(m)}(q) = p_{\theta^*}$ exists.
\end{theorem}

\begin{proof}
Since $u_j \in B_p$ and $q \in \mathcal{E}_p$ together 
imply $\mathbb{E}_q[|u_j|] < \infty$ by the Orlicz 
space structure of 
Subsection~\ref{subsec:orlicz_cramer}, the vector 
$\eta^q$ is well-defined and finite in $\mathbb{R}^d$. Since $\eta^q \in \mathcal{H}_{\eta} = \nabla\psi(\Theta)$  by hypothesis, there exists $\theta^* \in \Theta$  satisfying $\nabla\psi(\theta^*) = \eta^q$. By 
equation~\eqref{eq:mean_params}, this is equivalent 
to the moment matching 
condition~\eqref{eq:moment_matching}, and 
Theorem~\ref{thm:moment_matching} then gives 
$p_{\theta^*} = \pi^{(m)}(q)$.
\end{proof}
\begin{remark}[Boundary Behaviour]
\label{rem:boundary_behavior}
If $\eta^q$ lies on the boundary $\partial\mathcal{H}_{\eta}$, 
the infimum of $D_{\mathrm{KL}}(q \| p_\theta)$ over 
$\mathcal{N}$ is not attained by any $p_\theta \in 
\mathcal{N}$, and the optimal sequence of natural 
parameters diverges toward the boundary of $\Theta$. 
The hypothesis $\eta^q \in \mathcal{H}_{\eta}$ excludes 
this degenerate case.
\end{remark}

\begin{figure}[htbp]
\centering
\begin{tikzpicture}[scale=1.6, >=Stealth]
    \draw[thick, blue!60!black, fill=blue!5] 
        (0,0) coordinate (A) 
        to[out=20, in=190] (3,0.2) 
        to[out=10, in=260] (6.0,2.5) 
        to[out=80, in=350] (4.5,4.8) 
        to[out=170, in=20] (0.5,4.5) 
        to[out=200, in=110] cycle;
    \node[blue!80!black, font=\large\bfseries, 
          anchor=north east] at (5.8, 4.5) {$\mathcal{E}_p$};
    \node[blue!60!black, font=\scriptsize, align=right, 
          anchor=north east] at (5.8, 4.1) 
          {Maximal\\Exponential\\Model};

    \draw[thick, red!70!black, fill=red!10] 
        (1.2,1.0) coordinate (N1)
        to[out=10, in=170] (4.2,1.2) coordinate (N2)
        to[out=100, in=280] (4.4,2.4) coordinate (N3)
        to[out=190, in=10] (1.4,2.2) coordinate (N4)
        to[out=260, in=80] cycle;
    \node[red!80!black, font=\large\bfseries] 
          at (3.8, 2.3) {$\mathcal{N}$};
    \node[red!70!black, font=\tiny, align=center] 
          at (3.8, 2.0) {Parametric\\Submodel};

    \filldraw[black] (1.8, 3.8) circle (1.2pt) 
              coordinate (q);
    \node[above left=2pt of q, font=\large] {$q$};
    \node[above right=2pt of q, font=\scriptsize, gray] 
          {$\eta^q = (\mathbb{E}_q[u_j])_{j=1}^d$};
    \filldraw[red!70!black] (2.5, 1.8) circle (1.2pt) 
              coordinate (proj);
    \node[below right=2pt of proj, font=\small, 
          red!70!black, align=left] 
          {$p_{\theta^*} = \pi^{(m)}(q)$};
    \node[below=14pt of proj, font=\scriptsize, 
          red!70!black] 
          {$\eta(\theta^*) = (\mathbb{E}_{p_{\theta^*}}[u_j])_{j=1}^d$};
    \draw[->, thick, dashed, purple!70!black] (q) -- (proj) 
        node[midway, left=4pt, font=\tiny, 
             text=purple!80!black, align=right] 
             {$m$-projection\\(KL Minimization)};

    \node[draw=green!50!black, fill=white, rounded corners, 
          font=\scriptsize, align=center, line width=0.6pt] 
          (match) at (-1.0, 3.0) 
          {Moment Matching:\\$\eta(\theta^*) = \eta^q$};
    \draw[->, green!50!black, thin] (match.south) 
          to[out=270, in=180] (proj);

    \filldraw[blue!70!black] (5.2, 1.4) circle (1.2pt) 
              coordinate (ref);
    \node[right=2pt of ref, font=\small, 
          blue!70!black] {$p$};
    \node[below=2pt of ref, font=\tiny, 
          blue!60!black, align=center] 
          {Reference\\Density};

    \node[anchor=south west, font=\tiny, gray, align=left] 
          at (0.3, 0.3) {
        $\theta$: natural parameters\\
        $\eta$: mean parameters
    };
\end{tikzpicture}
\caption{The $m$-projection $\pi^{(m)} \colon \mathcal{E}_p \to \mathcal{N}$ minimizes $D_{\mathrm{KL}}(q \| p_\theta)$ 
over $p_\theta \in \mathcal{N}$. The optimal projection 
$p_{\theta^*}$ is uniquely characterized by the moment 
matching condition $\eta(\theta^*) = \eta^q$: the mean 
parameters of $p_{\theta^*}$ under $\mathcal{N}$ equal 
the moment vector of $q$ with respect to the sufficient 
statistics $(u_1,\ldots,u_d)$.}
\label{fig:m_projection_geometry}
\end{figure}
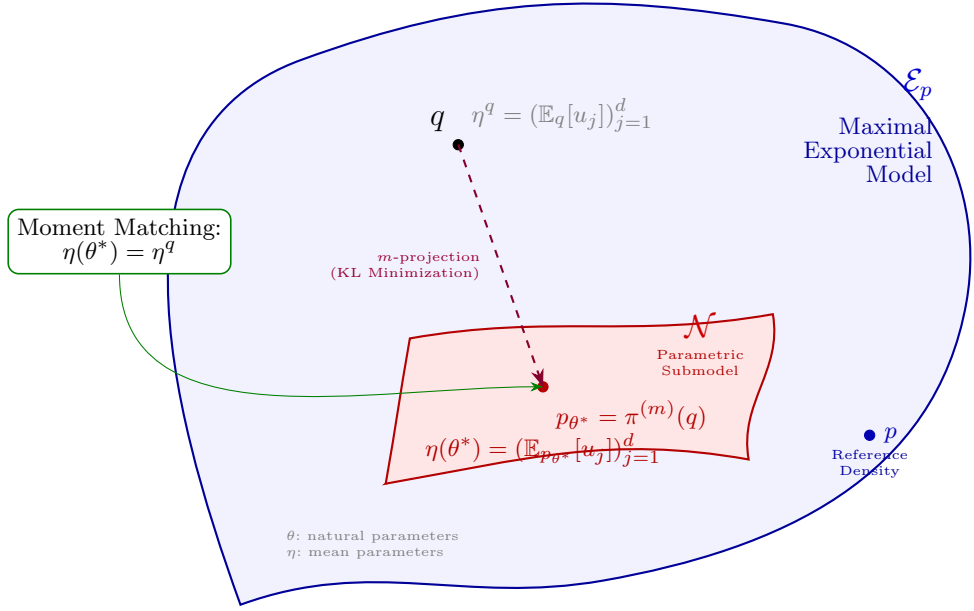

\begin{remark}[Connection to the Embedding $\mathcal{F}$]
\label{rem:projection_to_embedding}
The moment vector $\eta^q$, which satisfies  $\eta^q = \nabla\psi(\theta^*)$ at the projection point, 
is precisely the input through which $\mathcal{F}$ 
receives information about $q$; the embedding 
$\mathcal{F}\colon\mathcal{N}\to\mathrm{SPD}(d+1)$ 
constructed in Section~\ref{sec:spd_embedding} 
depends on $q \in \mathcal{E}_p$ exclusively via 
$\theta^* = (\nabla\psi)^{-1}(\eta^q)$.
\end{remark}

\subsection{The Many-to-One Nature of \texorpdfstring{$m$}{m}-Projection}
\label{subsec:many_to_one}
The $m$-projection reduces an infinite-dimensional 
distribution to a finite-dimensional moment vector. 
This dimensional reduction is unavoidably non-injective.

\begin{proposition}
\label{prop:many_to_one}
The $m$-projection $\pi^{(m)} \colon \mathcal{E}_p \to 
\mathcal{N}$ is not injective. For any $p_\theta \in 
\mathcal{N}$, there exist infinitely many distinct 
$q_1, q_2 \in \mathcal{E}_p$ with $q_1 \neq q_2$ such 
that $\pi^{(m)}(q_1) = \pi^{(m)}(q_2) = p_\theta$.
\end{proposition}

\begin{proof}
By Theorem~\ref{thm:moment_matching}, $\pi^{(m)}(q) = p_\theta$ 
if and only if $\mathbb{E}_q[u_j] = \eta_j(\theta)$ for all 
$j = 1,\ldots,d$. This condition depends on $q$ only through 
the $d$-dimensional moment vector $\eta^q \in \mathbb{R}^d$, 
leaving all higher-order structure of $q$ unconstrained. 
The fiber
\begin{equation}
(\pi^{(m)})^{-1}(p_\theta)
= \bigl\{ q \in \mathcal{E}_p : \mathbb{E}_q[u_j] = \eta_j(\theta),\;
j = 1,\ldots,d \bigr\}
\end{equation}
is cut out by only $d$ real-valued constraints on the 
infinite-dimensional space $\mathcal{E}_p$. Since 
$\mathcal{E}_p$ is infinite-dimensional, this fiber contains 
infinitely many distinct elements.
\end{proof}

\subsection{The Induced Pseudo-Metric on \texorpdfstring{$\mathcal{E}_p$}{Ep}}
\label{subsec:pseudo_metric}

The many-to-one nature of $\pi^{(m)}$ established in 
Proposition~\ref{prop:many_to_one} has a direct 
consequence for the distance structure on $\mathcal{E}_p$. 
Since the $m$-projection collapses infinitely many 
distributions to the same point of $\mathcal{N}$, any 
distance function on $\mathcal{N}$ pulled back to 
$\mathcal{E}_p$ via $\pi^{(m)}$ cannot be a true metric, 
as it assigns distance zero to distinct distributions 
that share the same moment vector. The resulting structure 
is a pseudo-metric, defined as follows.

\begin{definition}[Pseudo-Metric on $\mathcal{E}_p$]
\label{def:projected_distance}
Let $\varrho$ be any metric on $\mathcal{N}$. The {projected pseudo-metric} on 
$\mathcal{E}_p$ induced by $\pi^{(m)}$ and $\varrho$ is
\begin{equation}
\label{eq:projected_distance}
\tilde{d}(q_1, q_2) := \varrho\!\left(\pi^{(m)}(q_1),\, \pi^{(m)}(q_2)\right), \quad q_1, q_2 \in \mathcal{E}_p.
\end{equation}
\end{definition}

\begin{remark}
\label{rem:metric_choice}
The definition above holds for any metric $\varrho$ on $\mathcal{N}$. In the present framework the natural choice is the Riemannian distance on $\mathcal{N}$ induced by the pullback of the AIRM on $\mathrm{SPD}(d+1)$ via the embedding constructed in Section~\ref{sec:spd_embedding}. That specific metric, denoted $d_{\mathcal{F}}$, is defined and analyzed there; the properties of $\tilde{d}$ established in Proposition~\ref{prop:pseudo_metric} hold for any 
choice of $\varrho$ and in particular for $d_{\mathcal{F}}$.
\end{remark}

\begin{proposition}[Pseudo-Metric Properties]
\label{prop:pseudo_metric}
For any metric $\varrho$ on $\mathcal{N}$, the projected 
distance $\tilde{d} \colon \mathcal{E}_p \times \mathcal{E}_p 
\to [0,\infty)$ satisfies:
\begin{enumerate}[label=(\roman*)]
    \item \textbf{Non-negativity:} $\tilde{d}(q_1,q_2) 
    \geq 0$ for all $q_1, q_2 \in \mathcal{E}_p$.
    \item \textbf{Symmetry:} $\tilde{d}(q_1,q_2) = 
    \tilde{d}(q_2,q_1)$.
    \item \textbf{Triangle inequality:} 
    $\tilde{d}(q_1,q_3) \leq \tilde{d}(q_1,q_2) + 
    \tilde{d}(q_2,q_3)$.
    \item \textbf{Degeneracy:} $\tilde{d}(q_1,q_2) = 0$ 
    if and only if $\mathbb{E}_{q_1}[u_j] = 
    \mathbb{E}_{q_2}[u_j]$ for all $j = 1,\ldots,d$.
\end{enumerate}
In particular, $\tilde{d}$ is a pseudo-metric on 
$\mathcal{E}_p$; that is, it satisfies all metric axioms except 
that $\tilde{d}(q_1,q_2) = 0$ does not imply $q_1 = q_2$. 
\end{proposition}

\begin{proof}
Properties~(i), (ii), and (iii) follow directly from the 
corresponding properties of $\varrho$ on $\mathcal{N}$. For 
(i) and (ii): since $\varrho \geq 0$ and $\varrho$ is symmetric, 
$\tilde{d}(q_1,q_2) = \varrho(\pi^{(m)}(q_1),\pi^{(m)}(q_2)) 
\geq 0$ and equals $\varrho(\pi^{(m)}(q_2),\pi^{(m)}(q_1)) = 
\tilde{d}(q_2,q_1)$. For~(iii) the triangle inequality 
for $\varrho$ on $\mathcal{N}$ gives
\begin{align*}
\tilde{d}(q_1,q_3) 
&= \varrho\!\left(\pi^{(m)}(q_1), \pi^{(m)}(q_3)\right) \\
&\leq \varrho\!\left(\pi^{(m)}(q_1), \pi^{(m)}(q_2)\right) 
+ \varrho\!\left(\pi^{(m)}(q_2), \pi^{(m)}(q_3)\right) \\
&= \tilde{d}(q_1,q_2) + \tilde{d}(q_2,q_3).
\end{align*}
For~(iv) $\tilde{d}(q_1,q_2) = 0$ if and only if 
$\varrho(\pi^{(m)}(q_1), \pi^{(m)}(q_2)) = 0$. Since $\varrho$ is a metric on $\mathcal{N}$, this holds if and only if 
$\pi^{(m)}(q_1) = \pi^{(m)}(q_2)$. By 
Theorem~\ref{thm:moment_matching}, this is equivalent to 
$\mathbb{E}_{q_1}[u_j] = \mathbb{E}_{q_2}[u_j]$ for all 
$j = 1,\ldots,d$. By Proposition~\ref{prop:many_to_one}, 
there exist distinct $q_1 \neq q_2$ satisfying this 
condition, so $\tilde{d}(q_1,q_2) = 0$ does not imply $q_1 = q_2$.
\end{proof}

\section{Embedding of the Parametric Exponential Family 
into the Manifold of SPD Matrices}
\label{sec:spd_embedding}

The $m$-projection established in Section~\ref{sec:projection} 
reduces an arbitrary distribution $q \in \mathcal{E}_p$ to a 
finite-dimensional element $p_{\theta^*} \in \mathcal{N}$, 
identified entirely by the moment vector $\eta^q$. The 
pseudo-metric $\tilde{d}$ on $\mathcal{E}_p$ introduced in 
Subsection~\ref{subsec:pseudo_metric} depends on a metric $\varrho$ on $\mathcal{N}$, which has not yet been specified. The 
present section constructs this metric by embedding $\mathcal{N}$ into $\mathrm{SPD}(d+1)$ via the map $\mathcal{F}$ and pulling back the AIRM established in Subsection~\ref{subsec:spd}.

The embedding maps each $\theta \in \Theta$ to the 
expected outer product of the augmented sufficient statistics 
vector $V(x) = (u_1(x), \ldots, u_d(x), 1)^\top \in 
\mathbb{R}^{d+1}$. The constant function $1$ is appended to 
$(u_1, \ldots, u_d)$ for two reasons. The bottom-right entry 
of the resulting matrix is fixed at $1$ for every $\theta \in \Theta$, which will be used to detect 
when the AIRM geodesic exits the embedded submanifold in 
Subsection~\ref{subsec:geodesic_properties}. The last column 
of the matrix is exactly the mean parameter vector 
$\eta(\theta)$, separating first-moment information from the 
second-moment block and making the block structure of the 
embedding matrix explicit. The resulting $(d+1) \times (d+1)$ 
matrix encodes both the first and second moments of the 
sufficient statistics and is shown to be symmetric positive 
definite, placing it in $\mathrm{SPD}(d+1)$. The full geometry of this two-stage map is illustrated in Figure~\ref{fig:full_embedding_flow}.

\begin{theorem}[SPD Embedding]
\label{thm:spd_embedding}
Let $u_1, \ldots, u_d \in B_p$ be linearly independent in
$L^2(p\cdot\mu)$, and define the augmented vector
\[
V(x) = (u_1(x), \ldots, u_d(x), 1)^\top \in \mathbb{R}^{d+1},
\]
where $V_j(x) = u_j(x)$ for $j = 1,\ldots,d$ and
$V_{d+1}(x) = 1$.
The map $\mathcal{F} \colon \mathcal{N} \to \mathrm{SPD}(d+1)$,
where $\mathcal{F}$ acts on $p_\theta \in \mathcal{N}$ via its
natural parameter $\theta \in \Theta$, defined by
\begin{equation}
\label{eq:embedding_def}
\mathcal{F}(p_\theta) := \mathbb{E}_{p_\theta}
\!\left[V(x)V(x)^\top\right] =: S(\theta)
\end{equation}
is well-defined, and $S(\theta) \in \mathrm{SPD}(d+1)$
for every $\theta \in \Theta$. The matrix $S(\theta)$
has the block structure
\begin{equation}
\label{eq:block_structure}
S(\theta) = \begin{pmatrix}
A(\theta) & \eta(\theta) \\
\eta(\theta)^\top & 1
\end{pmatrix},
\end{equation}
where $[A(\theta)]_{jk} := \mathbb{E}_{p_\theta}[u_j u_k]$
for $j,k = 1,\ldots,d$ is the second moment matrix,
$[\eta(\theta)]_j := \mathbb{E}_{p_\theta}[u_j]$ for
$j = 1,\ldots,d$ is the mean parameter vector introduced 
in~\eqref{eq:mean_params}, and the bottom-right entry equals 
$1$ since $V_{d+1} \equiv 1$.
\end{theorem}

\begin{proof}
The block structure~\eqref{eq:block_structure} follows by
computing each entry of $S(\theta) =
\mathbb{E}_{p_\theta}[V(x)V(x)^\top]$ directly.
For $j, k \in \{1, \ldots, d\}$,
\[
[S(\theta)]_{jk}
= \mathbb{E}_{p_\theta}[V_j(x)V_k(x)]
= \mathbb{E}_{p_\theta}[u_j(x)\,u_k(x)]
= [A(\theta)]_{jk}.
\]
Since $V_{d+1}(x) = 1$ for all $x$, the entries in the
last column satisfy
\[
[S(\theta)]_{j,d+1}
= \mathbb{E}_{p_\theta}[u_j(x) \cdot 1]
= \mathbb{E}_{p_\theta}[u_j(x)]
= [\eta(\theta)]_j,
\]
and $[S(\theta)]_{d+1,j} = [\eta(\theta)]_j$ by symmetry
of the outer product. The bottom-right entry is
$[S(\theta)]_{d+1,d+1} = \mathbb{E}_{p_\theta}[1] = 1$.
This establishes~\eqref{eq:block_structure}.

Symmetry of $S(\theta)$ is immediate since
$[S(\theta)]_{jk} =
\mathbb{E}_{p_\theta}[V_j(x)V_k(x)] =
\mathbb{E}_{p_\theta}[V_k(x)V_j(x)] =
[S(\theta)]_{kj}$.

For positive definiteness, let $\mathbf{v} =
(v_1, \ldots, v_{d+1})^\top \in \mathbb{R}^{d+1}$
be arbitrary. By linearity of expectation,
\begin{align*}
\mathbf{v}^\top S(\theta)\,\mathbf{v}
= \sum_{j,k=1}^{d+1} v_j v_k\,
\mathbb{E}_{p_\theta}[V_j(x)V_k(x)]
= \mathbb{E}_{p_\theta}\!\left[
\left(\sum_{j=1}^{d+1} v_j V_j(x)\right)^{\!2}\right]
= \mathbb{E}_{p_\theta}
\!\left[(\mathbf{v}^\top V(x))^2\right] \geq 0,
\end{align*}
establishing positive semi-definiteness. For strict
positivity, suppose $\mathbf{v}^\top S(\theta)\mathbf{v}
= 0$. Then $\mathbb{E}_{p_\theta}
[(\mathbf{v}^\top V(x))^2] = 0$. Since $p_\theta \in
\mathcal{E}_p$ implies $p_\theta$ is equivalent to $\mu$
(Definition~\ref{def:density_space}), the integrand
being non-negative implies
\[
\mathbf{v}^\top V(x) = v_1 u_1(x) + \cdots + v_d u_d(x) + v_{d+1} = 0 \quad \mu\text{-a.e.}
\]
To show that $\mathbf{v} = \mathbf{0}$, we leverage the centering property of the sufficient statistics. Taking the expectation of both sides with respect to the reference measure $p \cdot \mu$ yields
\[
\sum_{j=1}^d v_j \mathbb{E}_p[u_j] + v_{d+1}\mathbb{E}_p[1] = 0.
\]
Since $u_j \in B_p$, we have $\mathbb{E}_p[u_j] = 0$ for all $j = 1, \ldots, d$ by definition. Because $\mathbb{E}_p[1] = 1$, this simplifies directly to $v_{d+1} = 0$. Substituting this back into the linear combination leaves $\sum_{j=1}^d v_j u_j(x) = 0$ $\mu$-a.e. Since the set $\{u_1, \ldots, u_d\}$ is given to be linearly independent in $L^2(p\cdot\mu)$, it follows that $v_1 = \cdots = v_d = 0$. Thus $\mathbf{v} = \mathbf{0}$, establishing $S(\theta) \succ 0$, 
so $S(\theta) \in \mathrm{SPD}(d+1)$.
\end{proof}

\begin{remark}[Schur Complement Identity]
\label{rem:block_interpretation}
By identity~\eqref{eq:A_I_decomposition}, the Fisher
information matrix $I(\theta)$ is the Schur complement
of the bottom-right entry of $S(\theta)$:
\[
A(\theta) - \eta(\theta)\eta(\theta)^\top = I(\theta).
\]
By the block matrix inversion formula, $S(\theta)^{-1}$
can be expressed in terms of $I(\theta)^{-1}$,
$\eta(\theta)$, and the scalar entry $1$; this expression
is used directly in the proof of
Theorem~\ref{thm:induced_metric}.
\end{remark}

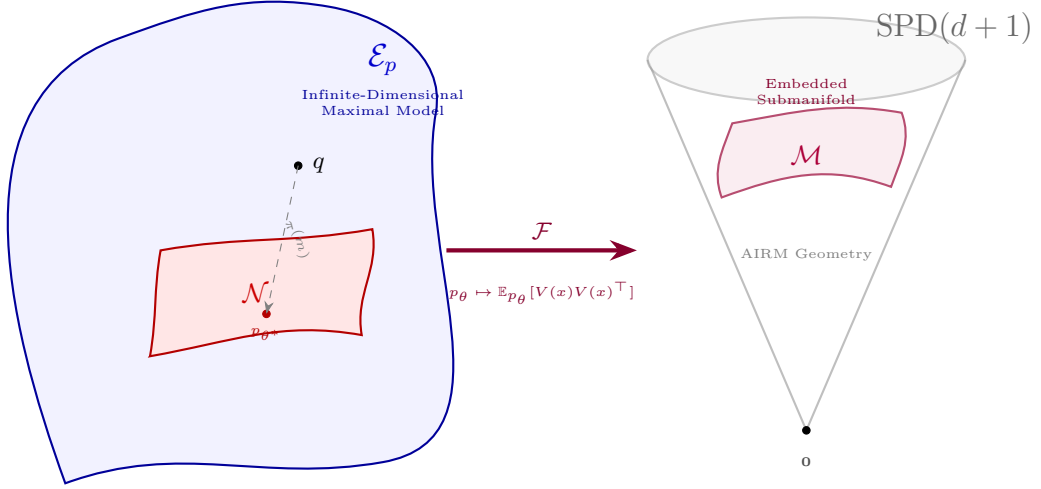
\begin{figure}[htbp]
\centering
\begin{tikzpicture}[scale=1.4, >=Stealth]
    \draw[thick, blue!60!black, fill=blue!5]
        (0,0) coordinate (A)
        to[out=20, in=190] (3,0.2)
        to[out=10, in=260] (3.5,3.5)
        to[out=80, in=350] (2.5,4.5)
        to[out=170, in=20] (0.5,4.2)
        to[out=200, in=110] cycle;
    \node[blue!80!black, font=\large\bfseries]
        at (3.0, 4.0) {$\mathcal{E}_p$};
    \node[blue!60!black, font=\tiny, align=center]
        at (3.0, 3.6) {Infinite-Dimensional\\Maximal Model};
    \draw[thick, red!70!black, fill=red!10]
        (0.8,1.2) coordinate (N1)
        to[out=10, in=170] (2.8,1.4) coordinate (N2)
        to[out=100, in=280] (2.9,2.4) coordinate (N3)
        to[out=190, in=10] (0.9,2.2) coordinate (N4)
        to[out=260, in=80] cycle;
    \node[red!80!black, font=\bfseries]
        at (1.8, 1.8) {$\mathcal{N}$};
    \filldraw[black] (2.2, 3.0) circle (1pt)
        coordinate (q);
    \node[right=2pt of q, font=\scriptsize] {$q$};
    \filldraw[red!70!black] (1.9, 1.6) circle (1pt)
        coordinate (proj);
    \node[below=2pt of proj, font=\tiny,
          red!70!black] {$p_{\theta^*}$};
    \draw[->, dashed, gray, thin] (q) -- (proj);
    \node[gray, font=\tiny, rotate=-75] at (2.2, 2.3)
        {$\pi^{(m)}$};
    \draw[->, ultra thick, purple!70!black]
        (3.6, 2.2) -- (5.4, 2.2)
        node[midway, above, font=\small\bfseries]
        {$\mathcal{F}$};
    \node[purple!70!black, font=\tiny, align=center]
        at (4.5, 1.8)
        {$p_\theta \mapsto
          \mathbb{E}_{p_\theta}[V(x)V(x)^\top]$};
    \begin{scope}[xshift=7cm, yshift=0.5cm]
        \draw[thick, gray!50] (0,0) -- (-1.5,3.5);
        \draw[thick, gray!50] (0,0) -- (1.5,3.5);
        \draw[thick, gray!80, fill=gray!10, opacity=0.5]
            (0,3.5) ellipse (1.5 and 0.4);
        \node[gray!80!black, font=\large\bfseries]
            at (1.4, 3.8) {$\mathrm{SPD}(d+1)$};
        \node[gray!60!black, font=\tiny] at (0, -0.3)
            {$\mathbf{0}$};
        \filldraw (0,0) circle (1pt);
        \draw[thick, purple!80!black, fill=purple!10,
              opacity=0.7]
            (-0.8, 2.2) coordinate (M1)
            to[out=20, in=160] (0.8, 2.3) coordinate (M2)
            to[out=70, in=290] (0.9, 3.0) coordinate (M3)
            to[out=170, in=10] (-0.7, 2.9) coordinate (M4)
            to[out=250, in=110] cycle;
        \node[purple!90!black, font=\bfseries]
            at (0, 2.6) {$\mathcal{M}$};
        \node[purple!70!black, font=\tiny, align=center]
            at (0, 3.2) {Embedded\\Submanifold};
        \node[anchor=north, font=\tiny, gray,
              align=center] at (0, 1.8)
            {AIRM Geometry};
    \end{scope}
\end{tikzpicture}
\caption{The $m$-projection $\pi^{(m)}$ sends
$q \in \mathcal{E}_p$ to its moment-matched element
$p_{\theta^*} \in \mathcal{N}$, identified by its
natural parameter $\theta^* \in \Theta$.
The embedding $\mathcal{F} \colon \mathcal{N} \to
\mathrm{SPD}(d+1)$ then maps $p_{\theta^*}$ to the SPD
matrix $S(\theta^*) = \mathbb{E}_{p_{\theta^*}}
[V(x)V(x)^\top]$. The image
$\mathcal{M} = \mathcal{F}(\mathcal{N}) \subset
\mathrm{SPD}(d+1)$ is a smooth submanifold
within the SPD cone, equipped with the geometry
induced by the AIRM.}
\label{fig:full_embedding_flow}
\end{figure}

The embedding $\mathcal{F}$ has been shown to map 
$\mathcal{N}$ into $\mathrm{SPD}(d+1)$. The next 
theorem establishes that $\mathcal{F}$ preserves the smooth 
structure of $\mathcal{N}$, making it 
a diffeomorphism onto its image $\mathcal{M} = 
\mathcal{F}(\mathcal{N})$.

\begin{theorem}[Diffeomorphism Property]
\label{thm:embedding_diffeomorphism}
The SPD embedding $\mathcal{F} \colon \mathcal{N} \to \mathrm{SPD}(d+1)$ 
is a diffeomorphism onto its image 
$\mathcal{M} := \mathcal{F}(\mathcal{N}) \subset \mathrm{SPD}(d+1)$. 
Specifically,
\begin{enumerate}[label=(\roman*)]
    \item $\mathcal{F}$ is injective.
    \item $\mathcal{F}$ is infinitely differentiable, 
    $\mathcal{F} \in C^\infty(\mathcal{N}, \mathrm{SPD}(d+1))$.
    \item The differential $D\mathcal{F}_{p_\theta} \colon 
    T_{p_\theta}\mathcal{N} \to T_{S(\theta)}\mathrm{SPD}(d+1)$ 
    has full rank $d$ at every $p_\theta \in \mathcal{N}$, 
    so $\mathcal{F}$ is an immersion.
    \item $\mathcal{M} = \mathcal{F}(\mathcal{N})$ is a 
    $d$-dimensional embedded submanifold of 
    $\mathrm{SPD}(d+1)$, and $\mathcal{F} \colon \mathcal{N} \to 
    \mathcal{M}$ is a diffeomorphism.
\end{enumerate}
\end{theorem}

\begin{proof}
    \textbf{(i).}
Suppose $S(\theta_1) = S(\theta_2)$ for some
$\theta_1, \theta_2 \in \Theta$. By the block
structure~\eqref{eq:block_structure}, equality of
the last columns gives $\eta(\theta_1) = \eta(\theta_2)$,
that is, $\nabla\psi(\theta_1) = \nabla\psi(\theta_2)$
by~\eqref{eq:mean_params}. Since $\nabla^2\psi(\theta)
= I(\theta) \succ 0$ for all $\theta \in \Theta$
by~\eqref{eq:fisher_matrix}, the function $\psi$ is
strictly convex on $\Theta$, so $\nabla\psi$ is strictly
monotone and hence injective. Therefore $\theta_1 = \theta_2$.

\textbf{(ii).}
It suffices to show that every entry of $S(\theta)$
is $C^\infty$ in $\theta$. By the block
structure~\eqref{eq:block_structure}, there are three
distinct blocks to consider. The bottom-right entry is the constant $1$, which is
trivially $C^\infty$. The last column and row entries are
$\eta_j(\theta) = \partial\psi/\partial\theta_j$
by~\eqref{eq:mean_params}, which are $C^\infty$
since $\psi \in C^\infty(\Theta)$~\cite{r9}. For the top-left block, the covariance identity gives
$[A(\theta)]_{jk} = [I(\theta)]_{jk} +
\eta_j(\theta)\eta_k(\theta)$.
Substituting~\eqref{eq:fisher_matrix}
and~\eqref{eq:mean_params} yields
\begin{equation}
\label{eq:A_smooth}
[A(\theta)]_{jk} =
\frac{\partial^2\psi}{\partial\theta_j\partial\theta_k}
(\theta) +
\frac{\partial\psi}{\partial\theta_j}(\theta)\,
\frac{\partial\psi}{\partial\theta_k}(\theta),
\end{equation}
a sum of a second partial derivative and a product of
first partial derivatives of $\psi$. Since
$\psi \in C^\infty(\Theta)$, each term
in~\eqref{eq:A_smooth} is $C^\infty$, and so is
their sum. Since all entries of $S(\theta)$ are
$C^\infty$, it follows that
$\mathcal{F} \in C^\infty(\mathcal{N}, \mathrm{SPD}(d+1))$.

\textbf{(iii).}
To prove that $\mathcal{F}$ is an immersion, we must show that
its differential $D\mathcal{F}_{p_{\theta}}$ is injective at every
$p_{\theta} \in \mathcal{N}$. The differential acts on a tangent vector
$v = (v_1, \ldots, v_d)^\top \in \mathbb{R}^d$ by
\begin{equation}
\label{eq:differential_action}
D\mathcal{F}_\theta(v) = \sum_{k=1}^d v_k
\frac{\partial S}{\partial\theta_k}(\theta)
\in \mathrm{Sym}(d+1).
\end{equation}
Suppose $D\mathcal{F}_{p_\theta}(v) = 0$. Since every entry of this
zero matrix must vanish, extracting the first $d$ entries of the
last column corresponding to the $(j, d+1)$ entries for
$j = 1, \ldots, d$ in the block
structure~\eqref{eq:block_structure} yields
\begin{equation}
\label{eq:last_col_condition}
\sum_{k=1}^d v_k
\frac{\partial\eta_j}{\partial\theta_k}(\theta) = 0,
\quad j = 1, \ldots, d.
\end{equation}
Since $\eta_j(\theta) = \frac{\partial\psi}{\partial\theta_j}(\theta)$
by~\eqref{eq:mean_params}, differentiating both sides with respect
to $\theta_k$ and invoking~\eqref{eq:fisher_matrix} gives
$\frac{\partial\eta_j}{\partial\theta_k}(\theta) = [I(\theta)]_{jk}$.
Substituting this identity into~\eqref{eq:last_col_condition} yields
$I(\theta)v = 0$. Since $I(\theta) \succ 0$
by~\eqref{eq:fisher_matrix}, the Fisher information matrix is
invertible, so $v = 0$. Consequently,
$\ker(D\mathcal{F}_\theta) = \{0\}$, proving injectivity.

\textbf{(iv).}
Parts~(i)--(iii) give that $\mathcal{F}\colon\mathcal{N}\to\mathcal{M}$
is an injective $C^\infty$ immersion, so a set-theoretic inverse
$\mathcal{F}^{-1}\colon\mathcal{M}\to\mathcal{N}$ exists. It remains
to show it is $C^\infty$.

Define $\rho\colon\mathcal{M}\to\mathcal{H}_{\eta}$ by extracting the
first $d$ entries of the last column of $S$,
\begin{equation}
\rho(S) := \bigl([S]_{1,d+1}, \ldots,
[S]_{d,d+1}\bigr)^\top \in \mathbb{R}^d.
\end{equation}
This is the restriction of a linear map on $\mathrm{Sym}(d+1)$, hence $C^\infty$.
For any
$S = S(\theta)\in\mathcal{M}$, the block
structure~\eqref{eq:block_structure} gives
$\rho(S) = \eta(\theta) = \nabla\psi(\theta)\in\mathcal{H}_{\eta}$.

Define the inverse as the composition
\begin{equation}
\mathcal{F}^{-1} := (\nabla\psi)^{-1}\circ\,\rho
\;\colon\; \mathcal{M} \to \mathcal{N}.
\end{equation}
The map $\nabla\psi\colon\Theta\to\mathcal{H}_{\eta}$ is $C^\infty$
and strictly monotone since $\nabla^2\psi(\theta) =
I(\theta)\succ 0$ for all $\theta\in\Theta$
by~\eqref{eq:fisher_matrix}, making it a $C^\infty$ bijection with everywhere
invertible Jacobian $I(\theta)$, so its inverse
$(\nabla\psi)^{-1}\colon\mathcal{H}_{\eta}\to\Theta$
is also $C^\infty$.

To verify $\mathcal{F}^{-1}$ inverts $\mathcal{F}$,
take any $\theta\in\Theta$,
\begin{align}
\mathcal{F}^{-1}(\mathcal{F}(\theta))
&= (\nabla\psi)^{-1}(\rho(S(\theta)))
&&\text{definition of } \mathcal{F}^{-1} \nonumber\\
&= (\nabla\psi)^{-1}(\eta(\theta))
&&\rho(S(\theta)) = \eta(\theta)     \text{ by block structure~\eqref{eq:block_structure}}
\nonumber\\
&= (\nabla\psi)^{-1}(\nabla\psi(\theta))
&&\eta(\theta) = \nabla\psi(\theta)
\text{ by~\eqref{eq:mean_params}} \nonumber\\
&= \theta.
&&(\nabla\psi)^{-1} \text{ inverts } \nabla\psi
\end{align}
Therefore $\mathcal{F}\colon\mathcal{N}\to\mathcal{M}$ is a
$C^\infty$ diffeomorphism, and $\mathcal{M} =
\mathcal{F}(\mathcal{N})$ is a $d$-dimensional embedded
submanifold of $\mathrm{SPD}(d+1)$.

\end{proof}

\subsection{The Induced Metric Structure}
\label{subsec:induced_metric}

The diffeomorphism $\mathcal{F}\colon\mathcal{N}\to\mathcal{M}$
established in Theorem~\ref{thm:embedding_diffeomorphism}
allows the AIRM on $\mathrm{SPD}(d+1)$
(Subsection~\ref{subsec:spd}) to be pulled back to a Riemannian metric on $\Theta$.

\begin{definition}[Pullback Metric]
\label{def:pullback_metric}
The {pullback metric} $g^\mathcal{F}$ on $\Theta$ induced 
by the AIRM~\eqref{eq:airm} via $\mathcal{F}$ is defined at 
$\theta \in \Theta$ by
\begin{equation}
g^{\mathcal{F}}_{jk}(\theta) = \frac{1}{2}\,\mathrm{tr}\!\left( 
S(\theta)^{-1} \frac{\partial S}{\partial\theta_j} 
S(\theta)^{-1} \frac{\partial S}{\partial\theta_k} 
\right), \quad j, k = 1, \ldots, d,
\end{equation}
where $S(\theta) = \mathcal{F}(\theta)$. The induced line 
element on $\mathcal{M} = \mathcal{F}(\mathcal{N})$ is
\begin{equation}
\label{eq:pullback_line_element}
ds^2_{\mathcal{M}} = \sum_{j,k=1}^d g^\mathcal{F}_{jk}(\theta) 
\, d\theta_j \, d\theta_k = \frac{1}{2}\,\mathrm{tr}
\!\left((S^{-1}\,dS)^2\right).
\end{equation}
\end{definition}

\begin{theorem}[Pullback Metric Formula]
\label{thm:induced_metric}
The differentials of the blocks of $S(\theta)$
with respect to $\theta_j$ are
\begin{align}
\left[\frac{\partial A}{\partial\theta_j}\right]_{lm}
&= \mathbb{E}_{p_\theta}[u_j u_l u_m]
- \mathbb{E}_{p_\theta}[u_j]\,
\mathbb{E}_{p_\theta}[u_l u_m],
\label{eq:dA_formula}\\
\frac{\partial\eta}{\partial\theta_j}
&= I(\theta)\,\mathbf{e}_j,
\label{eq:deta_formula}
\end{align}
where $\mathbf{e}_j$ is the $j$-th standard basis
vector in $\mathbb{R}^d$, so that
$d\eta = I(\theta)\,d\theta$. The pullback metric
$g^{\mathcal{F}}$ on $\Theta$
(Definition~\ref{def:pullback_metric}) has line element
\begin{equation}
\label{eq:induced_metric_formula}
ds^2_{\mathcal{M}} = \frac{1}{2}\,\mathrm{tr}
(I^{-1}\,dA\,I^{-1}\,dA)
- 2\,\eta^\top I^{-1}\,dA\,I^{-1}\,d\eta
+ (\eta^\top I^{-1}\,d\eta)^2
+ \gamma\,(d\eta^\top I^{-1}\,d\eta),
\end{equation}
where $\gamma = 1 + \eta^\top I^{-1}\eta$ and
$dA = \sum_j \frac{\partial A}{\partial\theta_j}d\theta_j$.
\end{theorem}

\begin{proof}
By identity~\eqref{eq:A_I_decomposition}, the Schur
complement of the bottom-right entry of $S(\theta)$
is $I(\theta)$, so the block inversion formula gives
\begin{equation}
\label{eq:S_inverse_block}
S^{-1} = \begin{pmatrix}
I^{-1} & -I^{-1}\eta \\
-\eta^\top I^{-1} & 1 + \eta^\top I^{-1}\eta
\end{pmatrix}.
\end{equation}
Setting $G := I^{-1}$ and $\gamma := 1 + \eta^\top
G\eta$, the score identity
\begin{equation}
\label{eq:score_identity}
\frac{\partial}{\partial\theta_j}
\mathbb{E}_{p_\theta}[f]
= \mathbb{E}_{p_\theta}[u_j f] -
\mathbb{E}_{p_\theta}[u_j]\mathbb{E}_{p_\theta}[f]
\end{equation}
applied with $f = u_l u_m$
gives~\eqref{eq:dA_formula}, and applied with
$f = u_l$ gives
\[
\frac{\partial\eta_l}{\partial\theta_j}
= [A(\theta)]_{jl} - \eta_j\eta_l
= [I(\theta)]_{lj},
\]
where the last equality uses $I = A - \eta\eta^\top$
entry-wise, so $\partial\eta/\partial\theta_j =
I(\theta)\,\mathbf{e}_j$,
establishing~\eqref{eq:deta_formula} and
$d\eta = I(\theta)\,d\theta$.

The differential of $S(\theta)$ is
\begin{equation}
\label{eq:dS_block}
dS = \begin{pmatrix} dA & d\eta \\
d\eta^\top & 0 \end{pmatrix},
\end{equation}
since $[S]_{d+1,d+1} = 1$ is constant. The block
product $S^{-1}dS$ is
\begin{align}
S^{-1}dS
&= \begin{pmatrix} G & -G\eta \\
-\eta^\top G & \gamma \end{pmatrix}
\begin{pmatrix} dA & d\eta \\
d\eta^\top & 0 \end{pmatrix} \nonumber \\
&= \begin{pmatrix}
G\,dA - G\eta\,d\eta^\top & G\,d\eta \\
-\eta^\top G\,dA + \gamma\,d\eta^\top &
-\eta^\top G\,d\eta
\end{pmatrix} =:
\begin{pmatrix} P & Q \\ R & s \end{pmatrix},
\label{eq:Sinv_dS}
\end{align}
where $P = G\,dA - G\eta\,d\eta^\top \in
\mathbb{R}^{d\times d}$,
$Q = G\,d\eta \in \mathbb{R}^{d\times 1}$,
$R = -\eta^\top G\,dA + \gamma\,d\eta^\top \in
\mathbb{R}^{1\times d}$, and
$s = -\eta^\top G\,d\eta \in \mathbb{R}$.

Squaring and taking the trace, using
$\mathrm{tr}(QR) = \mathrm{tr}(RQ)$ by the cyclic
property, and $\mathrm{tr}(RQ) = RQ$ since
$RQ \in \mathbb{R}$ is a scalar,
\begin{equation}
\label{eq:trace_formula}
\mathrm{tr}\!\left((S^{-1}dS)^2\right) =
\mathrm{tr}(P^2) + 2\,RQ + s^2.
\end{equation}
Expanding $P^2 = (G\,dA - G\eta\,d\eta^\top)^2$
and applying the cyclic trace property
with symmetry of $G$ and $dA$,
\begin{equation}
\label{eq:trP2_final}
\mathrm{tr}(P^2)
= \mathrm{tr}(G\,dA\,G\,dA)
- 2\,d\eta^\top G\,dA\,G\eta
+ (d\eta^\top G\eta)^2.
\end{equation}
Substituting $Q = G\,d\eta$ and
$R = -\eta^\top G\,dA + \gamma\,d\eta^\top$,
\begin{equation}
\label{eq:trQR_final}
RQ = -\eta^\top G\,dA\,G\,d\eta
+ \gamma\,d\eta^\top G\,d\eta,
\end{equation}
and $s^2 = (\eta^\top G\,d\eta)^2$.
Substituting into~\eqref{eq:trace_formula} and
using the fact that both sides of
$d\eta^\top G\,dA\,G\eta = \eta^\top G\,dA\,G\,d\eta$
are scalars related by transposition and hence equal,
and similarly $(d\eta^\top G\eta)^2 =
(\eta^\top G\,d\eta)^2$,
\begin{equation}
\label{eq:trace_simplified}
\mathrm{tr}\!\left((S^{-1}dS)^2\right)
= \mathrm{tr}(G\,dA\,G\,dA)
- 4\,\eta^\top G\,dA\,G\,d\eta
+ 2(\eta^\top G\,d\eta)^2
+ 2\gamma\,d\eta^\top G\,d\eta.
\end{equation}
Dividing by $2$ and substituting $G = I^{-1}$
gives~\eqref{eq:induced_metric_formula}.
\end{proof}

\begin{remark}
\label{rem:fisher_relationship}
Since $d\eta = I(\theta)\,d\theta$
by~\eqref{eq:deta_formula}, the last term
in~\eqref{eq:induced_metric_formula} satisfies
\begin{align*}
d\eta^\top I(\theta)^{-1}\,d\eta
&= (I(\theta)\,d\theta)^\top I(\theta)^{-1}
(I(\theta)\,d\theta) \\
&= d\theta^\top I(\theta)\,d\theta
= ds^2_{\mathrm{Fisher}},
\end{align*}
where the Fisher--Rao line element on $\Theta$ is
defined as
\[
ds^2_{\mathrm{Fisher}} :=
d\theta^\top I(\theta)\,d\theta
= \sum_{j,k=1}^d I_{jk}(\theta)\,
d\theta_j\,d\theta_k.
\]
Substituting into~\eqref{eq:induced_metric_formula}
gives
\begin{equation}
\label{eq:ds2_vs_fisher}
ds^2_{\mathcal{M}} =
\gamma\,ds^2_{\mathrm{Fisher}}
+ \tfrac{1}{2}\,\mathrm{tr}(I^{-1}dA\,I^{-1}dA)
- 2\,\eta^\top I^{-1}dA\,I^{-1}d\eta
+ (\eta^\top I^{-1}d\eta)^2,
\end{equation}
where $\gamma = 1 + \eta^\top I^{-1}\eta \geq 1$.
The correction beyond $\gamma\,ds^2_{\mathrm{Fisher}}$
comprises two terms involving $dA$, whose entries
by~\eqref{eq:dA_formula} encode third-order moments
of the sufficient statistics, and term $(\eta^\top I^{-1}d\eta)^2 = 
(\eta^\top d\theta)^2$ depends on 
the first-moment and parameter differentials.

\end{remark}

\begin{lemma}
\label{lem:kappa3_cancellation}
If $\kappa_{jlm}(\theta) :=
\mathbb{E}_{p_\theta}[(u_j - \eta_j)(u_l - \eta_l)
(u_m - \eta_m)] = 0$ for all $j, l, m$ and all
$\theta \in \Theta$, then
$ds^2_{\mathcal{M}} = ds^2_{\mathrm{Fisher}}$.
\end{lemma}

\begin{proof}
Expanding the product in $\kappa_{jlm}$ and taking
expectations, using $[A(\theta)]_{jl} =
[I(\theta)]_{jl} + \eta_j\eta_l$
from~\eqref{eq:A_I_decomposition}, gives
\begin{equation}
\label{eq:kappa_expansion}
\kappa_{jlm} = \mathbb{E}_{p_\theta}[u_j u_l u_m]
- \eta_m I_{jl} - \eta_l I_{jm}
- \eta_j I_{lm} - \eta_j\eta_l\eta_m.
\end{equation}
Subtracting $\eta_j\mathbb{E}_{p_\theta}[u_l u_m]
= \eta_j(I_{lm} + \eta_l\eta_m)$ from
$\mathbb{E}_{p_\theta}[u_j u_l u_m]$ as
in~\eqref{eq:dA_formula} and
substituting~\eqref{eq:kappa_expansion} gives
\begin{equation}
\label{eq:cov_expansion}
\left[\frac{\partial A}{\partial\theta_j}
\right]_{lm}
= \kappa_{jlm}(\theta)
+ \eta_m(\theta)\,I_{jl}(\theta)
+ \eta_l(\theta)\,I_{jm}(\theta).
\end{equation}
When $\kappa_{jlm} = 0$, summing over $j$ with
weights $d\theta_j$ and using symmetry of $I(\theta)$
together with $d\eta = I(\theta)\,d\theta$
from~\eqref{eq:deta_formula},
\begin{equation}
\label{eq:dA_kappa0}
[dA]_{lm}
= \sum_j (\eta_m I_{jl} + \eta_l I_{jm})d\theta_j
= \eta_m[d\eta]_l + \eta_l[d\eta]_m,
\end{equation}
which in matrix form is
$dA = d\eta\,\eta^\top + \eta\,d\eta^\top$,
the differential of the outer product $\eta\eta^\top$
by the product rule.

Setting $G = I^{-1}$,
$\alpha = \eta^\top G\eta$ (so that $\gamma = 1 + \alpha$),
$\beta = \eta^\top G\,d\eta$, and
$\delta = d\eta^\top G\,d\eta =
ds^2_{\mathrm{Fisher}}$, and noting
$G\,dA = G\,d\eta\,\eta^\top + G\eta\,d\eta^\top$,
we evaluate the three correction terms
of~\eqref{eq:ds2_vs_fisher} in turn.

For the first term, expanding
$\mathrm{tr}(G\,dA\,G\,dA)$ into four summands
and applying the cyclic trace property to each,
\begin{align*}
\mathrm{tr}(G\,dA\,G\,dA)
&= \mathrm{tr}(G\,d\eta\,\beta\,\eta^\top)
+ \mathrm{tr}(G\,d\eta\,\alpha\,d\eta^\top)
+ \mathrm{tr}(G\eta\,\delta\,\eta^\top)
+ \mathrm{tr}(G\eta\,\beta\,d\eta^\top) \\
&= \beta^2 + \alpha\delta + \alpha\delta + \beta^2
= 2\beta^2 + 2\alpha\delta,
\end{align*}
so $\tfrac{1}{2}\,\mathrm{tr}(G\,dA\,G\,dA)
= \beta^2 + \alpha\delta$.

For the second term, since
$G\,dA\,G\,d\eta = \beta\,G\,d\eta
+ \delta\,G\eta$,
\[
\eta^\top G\,dA\,G\,d\eta
= \beta\,\underbrace{\eta^\top G\,d\eta}_{\beta}
+ \delta\,\underbrace{\eta^\top G\eta}_{\alpha}
= \beta^2 + \alpha\delta,
\]
so $-2\,\eta^\top G\,dA\,G\,d\eta
= -2(\beta^2 + \alpha\delta)$.

For the third term, $(\eta^\top G\,d\eta)^2
= \beta^2$.

Summing the three corrections,
\[
(\beta^2 + \alpha\delta)
- 2(\beta^2 + \alpha\delta)
+ \beta^2
= -\alpha\delta
= -(\gamma - 1)\,\delta.
\]
Substituting into~\eqref{eq:ds2_vs_fisher},
\[
ds^2_{\mathcal{M}}
= \gamma\,\delta - (\gamma - 1)\,\delta
= \delta
= ds^2_{\mathrm{Fisher}}.
\]
\end{proof}

\subsection{The Isometric Embedding Property}
\label{subsec:isometry}

By construction of the pullback metric
(Definition~\ref{def:pullback_metric}),
$\mathcal{F}$ is an isometry from
$(\mathcal{N}, g^{\mathcal{F}})$ onto
$(\mathcal{M}, \mathrm{AIRM}|_{\mathcal{M}})$.
For all $p_\theta \in \mathcal{N}$ and
$v, w \in T_{p_\theta}\mathcal{N} \cong \mathbb{R}^d$,
\begin{equation}
\label{eq:isometry_condition}
g^{\mathcal{F}}_{p_\theta}(v, w) =
\langle D\mathcal{F}_{p_\theta}(v),
D\mathcal{F}_{p_\theta}(w)
\rangle_{S(\theta)}^{\mathrm{AIRM}}.
\end{equation}

\begin{remark}[Calvo--Oller Decomposition]
\label{rem:calvo_isometry}
Calvo and Oller~\cite{r25} established that for the
multivariate Gaussian family
\begin{equation}
\label{eq:calvo_oller_isometry}
(\mathcal{N}, I(\theta)) \cong
(\mathcal{M}, \mathrm{AIRM}|_{\mathcal{M}}).
\end{equation}
Within the present framework this isometry factors as
\begin{equation}
\label{eq:isometry_decomposition}
(\mathcal{N}, I(\theta))
\xrightarrow{\;\mathrm{id}\;}
(\mathcal{N}, g^{\mathcal{F}})
\xrightarrow{\;\mathcal{F}\;}
(\mathcal{M}, \mathrm{AIRM}|_{\mathcal{M}}),
\end{equation}
where the second map is an isometry for every
exponential family by~\eqref{eq:isometry_condition},
and the first map is an isometry if and only if
$g^{\mathcal{F}} = I(\theta)$. By
Lemma~\ref{lem:kappa3_cancellation}, this holds
when $\kappa_{jlm}(\theta) = 0$ for all $j,l,m$.
For the Gaussian family with linear sufficient
statistics $u_k(x) = x_k$, the cumulants
$\kappa_{jlm}$ vanish identically, so both maps are
isometries and~\eqref{eq:calvo_oller_isometry}
holds. The explicit verification appears in
Subsection~\ref{subsec:example_gaussian}. 
\end{remark}

\subsection{Geodesic Properties of the Embedded
Submanifold}
\label{subsec:geodesic_properties}

\begin{definition}[Totally Geodesic Submanifold]
\label{def:totally_geodesic}
A submanifold $\mathcal{L} \subset \mathrm{SPD}(m)$
is {totally geodesic} if for any two points
$S_0, S_1 \in \mathcal{L}$, the unique AIRM geodesic
connecting them lies entirely within $\mathcal{L}$.
\end{definition}

\begin{theorem}[Non-Geodesic Property of $\mathcal{M}$]
\label{thm:non_geodesic}
The embedded submanifold $\mathcal{M} =
\mathcal{F}(\mathcal{N}) \subset \mathrm{SPD}(d+1)$
is not totally geodesic in
$(\mathrm{SPD}(d+1), \mathrm{AIRM})$.
\end{theorem}

\begin{proof}
Every $S\in\mathcal{M}$ satisfies $[S]_{d+1,d+1} = 1$
by~\eqref{eq:block_structure}. For any two distinct
$S_0, S_1\in\mathcal{M}$, let
$M = S_0^{-1/2}S_1S_0^{-1/2} \in \mathrm{SPD}(d+1)$
and $\gamma(t) = S_0^{1/2}M^tS_0^{1/2}$ be the unique
AIRM geodesic from $S_0$ to $S_1$. Define
\begin{equation}
\label{eq:f_convex}
f(t) := [\gamma(t)]_{d+1,d+1}
= e_{d+1}^\top S_0^{1/2}M^t S_0^{1/2}e_{d+1}
= v^\top M^t v,
\end{equation}
where $v := S_0^{1/2}e_{d+1}$. Let $M = Q\Lambda Q^\top$
be the eigendecomposition with
$\Lambda = \mathrm{diag}(\lambda_1,\ldots,\lambda_{d+1})$,
$\lambda_i > 0$, and set $w := Q^\top v$. Then
\begin{equation}
f(t) = w^\top \Lambda^t w = \sum_{i=1}^{d+1} w_i^2\lambda_i^t.
\end{equation}
Since $S_0,S_1\in\mathcal{M}$, $f(0) = 1$ and $f(1) = 1$.
Each $t\mapsto\lambda_i^t$ is convex, hence $f$ is convex,
giving $f(t) \leq 1$ for all $t\in[0,1]$.

Suppose $f(t_0) = 1$ for some $t_0\in(0,1)$. Convexity
with $f(0)=f(1)=1$ forces $f\equiv 1$ on $[0,1]$.
Differentiating twice
\[
f''(t) = \sum_{i=1}^{d+1} w_i^2(\log\lambda_i)^2\lambda_i^t
= 0 \quad \forall\, t\in[0,1].
\]
Since $\lambda_i^t > 0$, dividing each term by $\lambda_i^t$
gives $w_i^2(\log\lambda_i)^2 = 0$ for every $i$, hence
\begin{equation}
\label{eq:either_or}
w_i = 0 \quad \text{or} \quad \lambda_i = 1
\quad \forall\, i.
\end{equation}
In either case $\lambda_i^t w_i = w_i$, so
$\Lambda^t w = w$ for all $t$, giving
\[
M^t v = Q\Lambda^t w = Qw = QQ^\top v = v
\quad \forall\, t,
\]
where $QQ^\top = I$ since $Q$ is orthogonal.
In particular $Mv = v$, which gives
\[
S_0^{-1/2}S_1S_0^{-1/2}\cdot S_0^{1/2}e_{d+1}
= S_0^{1/2}e_{d+1}
\implies S_1e_{d+1} = S_0e_{d+1}.
\]
By~\eqref{eq:block_structure}, $Se_{d+1} =
(\eta(\theta)^\top,1)^\top$ for any $S(\theta)\in\mathcal{M}$,
so $S_1e_{d+1} = S_0e_{d+1}$ implies
$\eta(\theta_1) = \eta(\theta_0)$.
Since $\nabla^2\psi = I(\theta)\succ 0$
by~\eqref{eq:fisher_matrix}, $\nabla\psi$ is strictly
monotone and hence injective implies
$\theta_0 = \theta_1$ and $S_0 = S_1$,
contradicting distinctness. Therefore
$f(t) < 1$ for all $t\in(0,1)$,
and $\gamma$ exits $\mathcal{M}$.
\end{proof}

\begin{remark}
\label{rem:geodesic_forward}
The non-geodesic property of $\mathcal{M}$ is the
geometric reason why the ambient AIRM distance
$d_{\mathrm{AIRM}}$ provides a strict lower bound
for $d_{\mathcal{F}}$. The geodesic in
$\mathrm{SPD}(d+1)$ passes through the ambient
space rather than staying on $\mathcal{M}$, making
it strictly shorter than any curve confined to
$\mathcal{M}$. The fixed-mean slices
\begin{equation}
\label{eq:fixed_mean_slice}
\mathcal{S}_{\eta_0} := \left\{
\begin{pmatrix} A & \eta_0 \\ \eta_0^\top & 1
\end{pmatrix} \in \mathrm{SPD}(r+1)
: A \in \mathrm{SPD}(r)
\right\}, \quad \eta_0\in\mathbb{R}^r,
\end{equation}
are totally geodesic submanifolds of
$\mathrm{SPD}(r+1)$; this is established in
Section~\ref{sec:partial_embedding}. For the
Gaussian partial embedding, the AIRM geodesic
between two fixed-mean points stays within
$\mathcal{M}_J$, so the lower bound is achieved
with equality on these slices; see
Remark~\ref{rem:calvo_equality}.
\end{remark}

\subsection{Distance Bounds via the Embedding}
\label{subsec:distance_bounds}

The Riemannian distance on $\mathcal{N}$ induced by the
pullback metric $g^{\mathcal{F}}$ 
(Definition~\ref{def:pullback_metric}) is
\begin{equation}
d_{\mathcal{F}}(p_{\theta_1}, p_{\theta_2}) :=
\inf_{\gamma} \int_0^1
\sqrt{g^{\mathcal{F}}_{\gamma(t)}
(\dot{\gamma}(t), \dot{\gamma}(t))}\,dt,
\end{equation}
where the infimum is over smooth paths
$\gamma\colon[0,1]\to\Theta$ with
$\gamma(0)=\theta_1$, $\gamma(1)=\theta_2$. By~\eqref{eq:isometry_condition}, this equals
the intrinsic distance on $\mathcal{M}$ along
paths confined to the submanifold.

\begin{theorem}[Main Distance Inequality]
\label{thm:distance_bound}
For any $p_{\theta_1}, p_{\theta_2} \in \mathcal{N}$,
\begin{equation}
\label{eq:distance_inequality}
d_{\mathrm{AIRM}}(\mathcal{F}(\theta_1),
\mathcal{F}(\theta_2))
\leq d_{\mathcal{F}}(\theta_1, \theta_2),
\end{equation}
with equality if and only if the AIRM geodesic
connecting $\mathcal{F}(\theta_1)$ to
$\mathcal{F}(\theta_2)$ lies entirely within
$\mathcal{M}$.
\end{theorem}

\begin{proof}
The AIRM distance between $\mathcal{F}(\theta_1)$
and $\mathcal{F}(\theta_2)$ is the infimum of
path lengths over all smooth paths in
$\mathrm{SPD}(d+1)$
\begin{equation}
d_{\mathrm{AIRM}}(\mathcal{F}(\theta_1),
\mathcal{F}(\theta_2))
= \inf_{\gamma\subset\mathrm{SPD}(d+1)}
\int_0^1\sqrt{\langle\dot{\gamma}(t),
\dot{\gamma}(t)\rangle_{\gamma(t)}^{\mathrm{AIRM}}}
\,dt.
\end{equation}
By~\eqref{eq:isometry_condition}, the pullback
distance $d_{\mathcal{F}}(\theta_1,\theta_2)$
equals the infimum of the same integrand over
paths constrained to $\mathcal{M}$
\begin{equation}
d_{\mathcal{F}}(\theta_1,\theta_2)
= \inf_{\gamma\subset\mathcal{M}}
\int_0^1\sqrt{\langle\dot{\gamma}(t),
\dot{\gamma}(t)\rangle_{\gamma(t)}^{\mathrm{AIRM}}}
\,dt.
\end{equation}
Since $\mathcal{M}\subset\mathrm{SPD}(d+1)$,
every path confined to $\mathcal{M}$ is also
a path in $\mathrm{SPD}(d+1)$, so the
infimum over the larger set is smaller or equal
\begin{equation}
\inf_{\gamma\subset\mathrm{SPD}(d+1)}
\int_0^1\sqrt{\langle\dot{\gamma},
\dot{\gamma}\rangle^{\mathrm{AIRM}}}\,dt
\leq
\inf_{\gamma\subset\mathcal{M}}
\int_0^1\sqrt{\langle\dot{\gamma},
\dot{\gamma}\rangle^{\mathrm{AIRM}}}\,dt,
\end{equation}
establishing~\eqref{eq:distance_inequality}.
Since $(\mathrm{SPD}(d+1),\mathrm{AIRM})$ is a
Hadamard manifold~\cite{r21}, there exists a
unique length-minimizing geodesic between any
two points. Equality holds if and only if this
geodesic is confined to $\mathcal{M}$.
\end{proof}

\begin{remark}
\label{rem:airm_tractable}
The AIRM distance on the left side
of~\eqref{eq:distance_inequality} is computable
in closed form via~\eqref{eq:airm_distance},
while $d_{\mathcal{F}}$ generally requires
solving a geodesic boundary value problem on
$\mathcal{M}$. The inequality thus provides a
tractable lower bound for the intrinsic distance.
Equality characterization on fixed-mean slices
and the connection to the Fisher--Rao distance
are developed in
Sections~\ref{sec:partial_embedding}
and~\ref{sec:complete_framework}.
\end{remark}

\section{Partial Embedding Criteria}
\label{sec:partial_embedding}

The full embedding $\mathcal{F}\colon\mathcal{N}\to\mathrm{SPD}(d+1)$
maps each parameter to a $(d+1)\times(d+1)$ SPD matrix.
This section develops the theory of {partial embeddings},
which use only a subset of the sufficient statistics, yielding
smaller matrices while preserving injectivity under an algebraic
condition on the omitted statistics.
The main result (Theorem~\ref{thm:partial_sufficient}) establishes
that if the omitted statistics can be expressed as polynomials of
degree at most two in the retained statistics, then the partial
embedding remains injective. This criterion recovers the classical
Calvo--Oller embedding for the Gaussian family as the canonical
special case and provides a systematic framework for identifying
when dimension reduction is admissible.

\subsection{Motivation: The Multivariate Gaussian Case}
\label{subsec:partial_motivation}

The multivariate Gaussian family with mean $\mu\in\mathbb{R}^n$
and covariance $\Sigma\in\mathrm{SPD}(n)$ has sufficient
statistics comprising all linear and quadratic terms in
$x = (x_1,\ldots,x_n)^\top$, giving
$d = n + \binom{n+1}{2} = \frac{n(n+3)}{2}$
sufficient statistics and a full embedding matrix of size
$(d+1)\times(d+1)$.

The classical Calvo--Oller embedding retains only the
$n$ linear statistics $u_k(x) = x_k$ and discards
the quadratic statistics $x_ix_j$, producing the
$(n+1)\times(n+1)$ matrix
\begin{equation}
\label{eq:calvo_oller_matrix}
\mathcal{F}_J(\theta) =
\begin{pmatrix}\Sigma + \mu\mu^\top & \mu \\
\mu^\top & 1\end{pmatrix},
\end{equation}
which has the block structure of~\eqref{eq:block_structure}
with $[A_J(\theta)]_{ij} = \mathbb{E}_{p_\theta}[x_ix_j]
= \Sigma_{ij} + \mu_i\mu_j$ and
$[\eta_J(\theta)]_k = \mathbb{E}_{p_\theta}[x_k] = \mu_k$.
This dimension reduction is admissible because each
discarded statistic satisfies
\begin{equation}
\label{eq:gaussian_poly}
x_ix_j = u_i(x)\cdot u_j(x),
\end{equation}
a polynomial of degree exactly two in the retained
statistics. The general criterion for when such a
reduction preserves injectivity is established in
Theorem~\ref{thm:partial_sufficient}.

\subsection{The Partial Embedding Map}
\label{subsec:partial_definition}

We now formalize the notion of using a subset of
sufficient statistics for embedding.

\begin{definition}[Partial Embedding]
\label{def:partial_embedding}
Let $J \subseteq \{1, \ldots, d\}$ be a non-empty
subset with $|J| = r$, written as
$J = \{j_1, \ldots, j_r\}$ with $j_1 < \cdots < j_r$.
The {partial augmented vector} corresponding
to $J$ is
\begin{equation}
V_J(x) = (u_{j_1}(x), \ldots, u_{j_r}(x), 1)^\top
\in \mathbb{R}^{r+1},
\end{equation}
and the {partial embedding} is the map
\begin{equation}
\mathcal{F}_J \colon \mathcal{N} \to \mathrm{SPD}(r+1),
\qquad
\mathcal{F}_J(p_\theta) :=
\mathbb{E}_{p_\theta}\!\left[V_J(x)\,V_J(x)^\top
\right] =: S_J(\theta).
\end{equation}
When $J = \{1, \ldots, d\}$, the partial embedding
reduces to the full embedding $\mathcal{F}$ of
Theorem~\ref{thm:spd_embedding}.
\end{definition}

The partial embedding
matrix inherits the block structure
of~\eqref{eq:block_structure}
\begin{equation}
\label{eq:partial_block_structure}
S_J(\theta) = \begin{pmatrix}
A_J(\theta) & \eta_J(\theta) \\
\eta_J(\theta)^\top & 1
\end{pmatrix},
\end{equation}
where $A_J(\theta)$ is the $r\times r$ matrix with
entries
\begin{equation}
\label{eq:AJ_entries}
[A_J(\theta)]_{kl} :=
\mathbb{E}_{p_\theta}[u_k\,u_l],
\qquad k, l \in J,
\end{equation}
and $\eta_J(\theta)\in\mathbb{R}^r$ is the vector
with entries
\begin{equation}
\label{eq:etaJ_entries}
[\eta_J(\theta)]_k :=
\mathbb{E}_{p_\theta}[u_k],
\qquad k \in J.
\end{equation}
As in~\eqref{eq:A_I_decomposition}, the Schur
complement gives the partial Fisher information
matrix
\begin{equation}
\label{eq:partial_schur}
A_J(\theta) - \eta_J(\theta)\,\eta_J(\theta)^\top
= I_J(\theta),
\qquad
[I_J(\theta)]_{kl} =
\mathrm{Cov}_{p_\theta}(u_k, u_l),
\quad k, l \in J.
\end{equation}

\subsection{The Polynomial Criterion for Injectivity}
\label{subsec:polynomial_criterion}

\begin{theorem}[Polynomial Criterion for Partial Embedding Injectivity]
\label{thm:partial_sufficient}
Let $J \subsetneq \{1, \ldots, d\}$ with $|J| = r < d$. Suppose that
for each $j \notin J$, the statistic $u_j$ can be expressed as a
polynomial of degree at most two in the retained statistics
$(u_k)_{k \in J}$,
\begin{equation}
\label{eq:poly_representation}
u_j(x) = \sum_{k,l \in J} \alpha_{kl}^{(j)}\,u_k(x)\,u_l(x)
+ \sum_{k \in J} \beta_k^{(j)}\,u_k(x) + \gamma^{(j)}
\qquad \mu\text{-a.e.},
\end{equation}
for constants $\alpha_{kl}^{(j)}, \beta_k^{(j)}, \gamma^{(j)} \in
\mathbb{R}$. Then $\mathcal{F}_J \colon \mathcal{N} \to \mathrm{SPD}(r+1)$
is injective.
\end{theorem}

\begin{proof}
Suppose $S_J(\theta_1) = S_J(\theta_2)$ for some
$\theta_1, \theta_2 \in \Theta$. By the block
structure~\eqref{eq:partial_block_structure},
\begin{align}
\mathbb{E}_{p_{\theta_1}}[u_k] &=
\mathbb{E}_{p_{\theta_2}}[u_k]
\qquad \text{for all } k \in J,
\label{eq:first_moments_equal}\\
\mathbb{E}_{p_{\theta_1}}[u_k u_l] &=
\mathbb{E}_{p_{\theta_2}}[u_k u_l]
\qquad \text{for all } k, l \in J.
\label{eq:second_moments_equal}
\end{align}
Fix $j \notin J$. Since~\eqref{eq:poly_representation}
holds $\mu$-a.e.\ and $p_\theta$ is equivalent to $\mu$
(Definition~\ref{def:density_space}), taking expectations
under $p_\theta$ and using linearity gives
\begin{equation}
\label{eq:eta_j_recovery}
\mathbb{E}_{p_\theta}[u_j]
= \sum_{k,l \in J}
\alpha_{kl}^{(j)}\,\mathbb{E}_{p_\theta}[u_k u_l]
+ \sum_{k \in J} \beta_k^{(j)}\,\mathbb{E}_{p_\theta}[u_k]
+ \gamma^{(j)}.
\end{equation}
The right-hand side depends on $\theta$ only through
$\{[A_J(\theta)]_{kl}\}_{k,l\in J}$ and
$\{[\eta_J(\theta)]_k\}_{k\in J}$, which are equal
for $\theta_1$ and $\theta_2$
by~\eqref{eq:first_moments_equal}
and~\eqref{eq:second_moments_equal}. Therefore
$\mathbb{E}_{p_{\theta_1}}[u_j] =
\mathbb{E}_{p_{\theta_2}}[u_j]$ for every $j\notin J$.

Combined with~\eqref{eq:first_moments_equal},
$\eta(\theta_1) = \eta(\theta_2)$. Since
$\nabla^2\psi(\theta) = I(\theta) \succ 0$
by~\eqref{eq:fisher_matrix}, $\eta = \nabla\psi$
is injective, so $\theta_1 = \theta_2$.
\end{proof}

\begin{remark}
\label{rem:degree_two_threshold}
The degree bound of two is required because $S_J(\theta)$ encodes
moments of the retained statistics only up to degree two: degree-one
moments $\mathbb{E}_{p_\theta}[u_k]$ appear in $\eta_J(\theta)$
and degree-two moments $\mathbb{E}_{p_\theta}[u_k u_l]$ appear in
$A_J(\theta)$. The recovery identity~\eqref{eq:eta_j_recovery}
therefore relies strictly on this second-order information.
\end{remark}

\subsection{Application to Multivariate Gaussians}
\label{subsec:gaussian_partial}
For the multivariate gaussian partial embedding introduced in
Subsection~\ref{subsec:partial_motivation}, each
omitted quadratic statistic satisfies
$u_{ij}(x) = u_i(x)\cdot u_j(x)$,
a polynomial of degree exactly two in the retained
linear statistics. Theorem~\ref{thm:partial_sufficient}
therefore confirms that $\mathcal{F}_J\colon\mathcal{N}
\to\mathrm{SPD}(n+1)$ is injective, and the
embedding matrix takes the form
\begin{equation*}
\mathcal{F}_J(\theta) =
\begin{pmatrix}\Sigma + \mu\mu^\top & \mu \\
\mu^\top & 1\end{pmatrix},
\end{equation*}
recovering the Calvo--Oller embedding~\cite{r25}
(equation~\eqref{eq:calvo_oller_matrix}).

\begin{remark}[Beyond Gaussians]
\label{rem:beyond_gaussians}
The polynomial criterion applies to any exponential
family where the omitted statistics are expressible
as polynomials of degree at most two in the retained
ones. For families where no such representation exists, the full embedding $\mathcal{F}$ of
Theorem~\ref{thm:spd_embedding} is required; the Gamma family, where $u_1(x) = \log x$ and $u_2(x) = x$ are algebraically independent, is one such case
(Subsection~\ref{subsec:example_gamma}).
\end{remark}

\subsection{Totally Geodesic Fixed-Mean Slices}
\label{subsec:totally_geodesic_slices}

For the full embedding $\mathcal{F}$, fixing $\eta(\theta)=\eta_0$
determines $\theta$ uniquely since $\nabla\psi$ is injective,
so $\mathcal{M}\cap\mathcal{S}_{\eta_0}$ is a single point.
For the partial embedding $\mathcal{F}_J$, fixing
$\eta_J(\theta)=\eta_0$ constrains only $r < d$ mean
parameters, leaving the remaining $d-r$ free to vary. The intersection
$\mathcal{M}_J\cap\mathcal{S}_{\eta_0}$ therefore has positive
dimension, and the following result becomes non-trivial.

\begin{theorem}[Totally Geodesic Fixed-Mean Slices]
\label{thm:totally_geodesic}
For any fixed $\eta_0\in\mathbb{R}^r$, the set
\begin{equation}
\label{eq:fixed_mean_slice_J}
\mathcal{S}_{\eta_0} := \left\{
\begin{pmatrix} A & \eta_0 \\ \eta_0^\top & 1
\end{pmatrix} \in \mathrm{SPD}(r+1)
: A \in \mathrm{SPD}(r)
\right\}
\end{equation}
is a totally geodesic submanifold of
$(\mathrm{SPD}(r+1), d_{\mathrm{AIRM}})$.
\end{theorem}

\begin{proof}
Define the invertible matrix
\[
T = \begin{pmatrix} I_r & -\eta_0 \\ 0 & 1 \end{pmatrix},
\qquad
T^{-1} = \begin{pmatrix} I_r & \eta_0 \\ 0 & 1 \end{pmatrix}.
\]
For any $S = \begin{pmatrix} A & \eta_0 \\ \eta_0^\top & 1
\end{pmatrix} \in \mathcal{S}_{\eta_0}$, direct computation gives
\[
TST^\top =
\begin{pmatrix} A - \eta_0\eta_0^\top & 0 \\ 0 & 1
\end{pmatrix}.
\]
The map $S \mapsto TST^\top$ sends $\mathcal{S}_{\eta_0}$
bijectively onto the block-diagonal set
\[
\mathcal{D} := \left\{
\begin{pmatrix} P & 0 \\ 0 & 1 \end{pmatrix}
: P \in \mathrm{SPD}(r) \right\}.
\]
For any two points $\tilde{S}_0, \tilde{S}_1 \in \mathcal{D}$ with $\tilde{S}_i =
\begin{pmatrix} P_i & 0 \\ 0 & 1 \end{pmatrix}$ for $i=0,1$,
\[
\bigl(\tilde{S}_0^{-1/2}\tilde{S}_1\tilde{S}_0^{-1/2}\bigr)^t
= \begin{pmatrix}
(P_0^{-1/2}P_1 P_0^{-1/2})^t & 0 \\ 0 & 1
\end{pmatrix}, \quad t\in[0,1].
\]
The upper-left block is the AIRM geodesic between $P_0$ and
$P_1$ in $\mathrm{SPD}(r)$, which stays in $\mathrm{SPD}(r)$
for all $t\in[0,1]$. Hence the geodesic between $\tilde{S}_0$ and
$\tilde{S}_1$ stays in $\mathcal{D}$, so $\mathcal{D}$ is
totally geodesic in $\mathrm{SPD}(r+1)$. Since the AIRM is invariant under congruence
transformations, the map $S \mapsto T S T^\top$
is an isometry of $\mathrm{SPD}(r+1)$. Isometries
preserve totally geodesic submanifolds, and
$\mathcal{S}_{\eta_0} =
T^{-1} \mathcal{D} (T^{-1})^\top$, so
$\mathcal{S}_{\eta_0}$ is totally geodesic in
$\mathrm{SPD}(r+1)$.
\end{proof}

\begin{corollary}[Equality of distances on fixed-mean slices]
\label{cor:fixed_mean_equality}
For any $\theta_1, \theta_2$ with $\eta_J(\theta_1)=\eta_J(\theta_2)=\eta_0$, 
the AIRM geodesic between $\mathcal{F}_J(\theta_1)$ and $\mathcal{F}_J(\theta_2)$ 
lies entirely within $\mathcal{S}_{\eta_0}$ by Theorem~\ref{thm:totally_geodesic}. 
If this geodesic also lies within $\mathcal{M}_J$, then
\[
d_{\mathrm{AIRM}}\!\bigl(\mathcal{F}_J(\theta_1), \mathcal{F}_J(\theta_2)\bigr)
= d_{\mathcal{F}_J}(\theta_1, \theta_2).
\]
\end{corollary}

\begin{remark}[Calvo--Oller equality case]
\label{rem:calvo_equality}
For the Gaussian partial embedding with $J=\{1,\ldots,r\}$ and fixed mean $\mu=\mu_0$, 
the remaining free parameter is $\Sigma\in\mathrm{SPD}(r)$, and the AIRM geodesic 
between two fixed-mean Gaussian embeddings is a curve of the form
\[
\gamma(t) = \begin{pmatrix}\Sigma(t)+\mu_0\mu_0^\top & \mu_0 \\
\mu_0^\top & 1\end{pmatrix}, \quad t \in [0,1],
\]
where $\Sigma(t)$ is the AIRM geodesic in $\mathrm{SPD}(r)$ between $\Sigma_1$ and $\Sigma_2$. 
Since $\Sigma(t)\in\mathrm{SPD}(r)$ for all $t$, each $\gamma(t)$ equals $\mathcal{F}_J(\mu_0, \Sigma(t))$ 
and therefore lies in $\mathcal{M}_J$. The condition of Corollary~\ref{cor:fixed_mean_equality} 
is therefore satisfied, recovering the classical result of Calvo and Oller~\cite{r25} 
that $d_{\mathrm{AIRM}}=d_{\mathcal{F}_J}$ for fixed-mean Gaussians.
\end{remark}

\section{The Complete Distance Framework on
\texorpdfstring{$\mathcal{E}_p$}{Ep}}
\label{sec:complete_framework}

The two-stage construction introduced in
Section~\ref{sec:framework} is now complete.
The $m$-projection $\pi^{(m)}\colon\mathcal{E}_p
\to\mathcal{N}$ and the embedding
$\mathcal{F}\colon\mathcal{N}\to\mathcal{M}\subset
\mathrm{SPD}(d+1)$ together yield a pseudo-metric
$\tilde{d}$ on $\mathcal{E}_p$ with a closed-form
lower bound via the AIRM distance.

\subsection{The Two-Stage Construction}
\label{subsec:two_stage}

The framework operates through the composition
\begin{equation}
\label{eq:map_composition}
\mathcal{E}_p
\xrightarrow{\;\pi^{(m)}\;}
\mathcal{N}
\xrightarrow{\;\mathcal{F}\;}
\mathcal{M} \subset \mathrm{SPD}(d+1),
\end{equation}
where $\mathcal{F}$ acts on the natural parameter
coordinate $\theta$ of $p_\theta\in\mathcal{N}$.
The embedding $\mathcal{F}$ is a diffeomorphism
onto $\mathcal{M} = \mathcal{F}(\mathcal{N})$
(Theorem~\ref{thm:embedding_diffeomorphism}),
so pulling back the AIRM via $\mathcal{F}$ yields
the pullback metric $g^{\mathcal{F}}$ on $\mathcal{N}$
(Definition~\ref{def:pullback_metric}) and the
associated Riemannian distance $d_{\mathcal{F}}$.
By~\eqref{eq:isometry_condition}, $\mathcal{F}$ is
an isometry from $(\mathcal{N}, g^{\mathcal{F}})$ to
$(\mathcal{M}, \mathrm{AIRM}|_{\mathcal{M}})$,
so $d_{\mathcal{F}}(\theta_1, \theta_2)$ equals
the intrinsic Riemannian distance on $\mathcal{M}$
between $\mathcal{F}(\theta_1)$ and
$\mathcal{F}(\theta_2)$. Pulling $d_{\mathcal{F}}$
back through $\pi^{(m)}$ via
Definition~\ref{def:projected_distance} gives
the pseudo-metric $\tilde{d}$ on $\mathcal{E}_p$
(Proposition~\ref{prop:pseudo_metric}):
\begin{equation}
\label{eq:projected_distance_Ep}
\tilde{d}(q_1, q_2)
:= d_{\mathcal{F}}\!\bigl(\pi^{(m)}(q_1),\,
\pi^{(m)}(q_2)\bigr),
\qquad q_1, q_2 \in \mathcal{E}_p.
\end{equation}

\subsection{The Complete Distance Inequality}
\label{subsec:complete_inequality}

Combining Theorem~\ref{thm:distance_bound}
with~\eqref{eq:projected_distance_Ep} gives the
main result of this section.

\begin{theorem}[Complete Distance Inequality]
\label{thm:complete_inequality}
For any $q_1, q_2 \in \mathcal{E}_p$, let $\theta_i$
denote the natural parameters of the $m$-projections
$p_{\theta_i} = \pi^{(m)}(q_i)$ and let
$S_i = \mathcal{F}(\theta_i)$. Then
\begin{equation}
\label{eq:complete_chain}
d_{\mathrm{AIRM}}(S_1, S_2)
\leq \tilde{d}(q_1, q_2).
\end{equation}
Equality holds if and only if the AIRM geodesic
connecting $S_1$ and $S_2$ in $\mathrm{SPD}(d+1)$
lies entirely within $\mathcal{M} =
\mathcal{F}(\mathcal{N})$.
\end{theorem}

\begin{proof}
By~\eqref{eq:projected_distance_Ep},
\[
\tilde{d}(q_1, q_2)
= d_{\mathcal{F}}\!\bigl(\pi^{(m)}(q_1),
\pi^{(m)}(q_2)\bigr)
= d_{\mathcal{F}}(\theta_1, \theta_2).
\]
Theorem~\ref{thm:distance_bound} then gives
$d_{\mathrm{AIRM}}(S_1, S_2) \leq
d_{\mathcal{F}}(\theta_1, \theta_2)
= \tilde{d}(q_1, q_2)$, with equality if and
only if the unique AIRM geodesic between $S_1$
and $S_2$ lies entirely within $\mathcal{M}$.
\end{proof}

\section{Examples}
\label{sec:examples}

This section applies the framework of
Section~\ref{sec:complete_framework} through three
worked examples. Section~\ref{subsec:example_gaussian} treats
the univariate Gaussian family, where the polynomial criterion
of Theorem~\ref{thm:partial_sufficient} permits a partial
embedding that
recovers the Calvo--Oller embedding~\cite{r25} as a special
case, and the pullback metric coincides with the Fisher--Rao
metric by Lemma~\ref{lem:kappa3_cancellation}. 
Section~\ref{subsec:example_gamma} treats the Gamma family,
where the sufficient statistics are algebraically independent
so the full embedding is required, and non-vanishing third cumulants produce a
strict discrepancy between the pullback metric and the
Fisher--Rao metric. Section~\ref{subsec:hypothesis_testing}
applies the pseudo-metric $\tilde{d}$ to construct an
affine-invariant two-sample test, using the AIRM distance
between empirical SPD embedding matrices as the test
statistic, and compares its power against MMD and
Hotelling's $T^2$ via permutation on two distinct examples.

\subsection{Univariate Gaussian}
\label{subsec:example_gaussian}

Fix $p_0 = \mathcal{G}(0,1)$ as the reference density. The family
$\mathcal{G} = \{\mathcal{G}(\mu,\sigma^2) : \mu\in\mathbb{R},\,\sigma^2>0\}$
embeds into $\mathcal{E}_{p_0}$ via
\begin{equation}
\label{eq:gauss_exp}
p_{(\mu,\sigma^2)}
= \exp\!\bigl(\theta_1\,u_1 + \theta_2\,u_2 - \psi(\theta)\bigr)\,p_0,
\end{equation}
where the sufficient statistics, centred at $p_0$ so that
$\mathbb{E}_{p_0}[u_j]=0$ and $u_j\in B_{p_0}$, are
\begin{equation}
\label{eq:gauss_stats}
u_1(x) = x, \qquad u_2(x) = x^2 - 1,
\end{equation}
and the natural parameters are
\begin{equation}
\label{eq:gauss_params}
\theta_1 = \frac{\mu}{\sigma^2}, \qquad
\theta_2 = \frac{1}{2} - \frac{1}{2\sigma^2},
\qquad \Theta = \mathbb{R}\times\bigl(-\infty,\tfrac{1}{2}\bigr).
\end{equation}

\smallskip\noindent\textbf{Partial embedding via the polynomial criterion.}
Taking $J=\{1\}$ and retaining only $u_1$, the omitted statistic satisfies
\begin{equation}
\label{eq:gauss_poly}
u_2(x) = u_1(x)^2 - 1,
\end{equation}
which has the form~\eqref{eq:poly_representation} with
$\alpha_{11}^{(2)}=1$, $\beta_1^{(2)}=0$, $\gamma^{(2)}=-1$.
Theorem~\ref{thm:partial_sufficient} therefore guarantees that
$\mathcal{F}_J\colon\mathcal{N}\to\mathrm{SPD}(2)$ is injective.
Computing the blocks from~\eqref{eq:etaJ_entries}--\eqref{eq:AJ_entries},
\begin{equation}
\label{eq:gauss_blocks}
\eta_J(\theta) = \mathbb{E}_{p_\theta}[u_1] = \mu, \qquad
A_J(\theta) = \mathbb{E}_{p_\theta}[u_1^2] = \sigma^2 + \mu^2,
\end{equation}
so the embedding~\eqref{eq:partial_block_structure} gives
\begin{equation}
\label{eq:gauss_embed}
S_J(\mu,\sigma^2)
= \begin{pmatrix} \sigma^2+\mu^2 & \mu \\ \mu & 1 \end{pmatrix}
\in \mathrm{SPD}(2).
\end{equation}
The Schur complement~\eqref{eq:partial_schur} recovers the partial Fisher
information: $I_J = A_J - \eta_J^2 = \sigma^2$.

\smallskip\noindent\textbf{Pullback metric and Calvo--Oller recovery.}
For $X\sim\mathcal{G}(\mu,\sigma^2)$, the only third joint cumulant of
$u_1$ is
\begin{equation}
\label{eq:gauss_kappa}
\kappa_{111}(\theta)
= \mathbb{E}_{p_\theta}\!\bigl[(X-\mu)^3\bigr] = 0,
\end{equation}
since all odd central moments of a Gaussian vanish.
Lemma~\ref{lem:kappa3_cancellation} then gives $ds^2_{\mathcal{M}_J} =
ds^2_{\mathrm{Fisher}}$, which for the $(\mu,\sigma)$ parametrization is
\begin{equation}
\label{eq:gauss_pullback}
ds^2_{\mathcal{M}_J}
= \frac{(d\mu)^2}{\sigma^2} + \frac{2\,(d\sigma)^2}{\sigma^2}.
\end{equation}
Hence $g^{\mathcal{F}_J} = I_J(\theta)$, and by Remark~\ref{rem:calvo_isometry}
the embedding~\eqref{eq:gauss_embed} is the Calvo--Oller
embedding~\cite{r25}.

\subsection{Gamma Distribution}
\label{subsec:example_gamma}

Fix $p_0 = \mathrm{Gamma}(1,1)$ as the reference density,
$p_0(x) = e^{-x}$. The family
$\mathcal{K} = \{\mathrm{Gamma}(\alpha,\beta) : \alpha,\beta > 0\}$
has density $p(x;\alpha,\beta) =
\frac{\beta^\alpha}{\Gamma(\alpha)}x^{\alpha-1}e^{-\beta x}$,
giving the ratio
\begin{equation}
\frac{p(x;\alpha,\beta)}{p_0(x)}
= x^{\alpha-1}\,e^{(1-\beta)x}\,
\frac{\beta^\alpha}{\Gamma(\alpha)},
\end{equation}
so that
\begin{equation}
\label{eq:gamma_exp}
p(x;\alpha,\beta)
= \exp\!\bigl[(\alpha-1)\log x + (1-\beta)\,x
- \psi(\theta)\bigr]\,p_0(x),
\end{equation}
where $\psi(\theta) = \log\Gamma(\alpha) - \alpha\log\beta$
is the log-partition function. This gives the natural
parameters
\begin{equation}
\label{eq:gamma_params}
\theta_1 = \alpha - 1, \qquad \theta_2 = 1-\beta,
\qquad \Theta = (-1,\infty)\times(-\infty,1),
\end{equation}
with sufficient statistics $\log x$ and $x$ respectively.
Centering at $p_0$ via $\mathbb{E}_{p_0}[\log x] = 
\Psi^{(0)}(1) = -\gamma_E$ and $\mathbb{E}_{p_0}[x] = 1$
gives the centered sufficient statistics
\begin{equation}
\label{eq:gamma_stats}
u_1(x) = \log x + \gamma_E, \qquad u_2(x) = x - 1,
\end{equation}
where $\Psi^{(k)}(\alpha) := \frac{d^{k+1}}{d\alpha^{k+1}}
\log\Gamma(\alpha)$ is the $k$-th polygamma function and
$\gamma_E = -\Psi^{(0)}(1)$ is the Euler--Mascheroni
constant~\cite{r61}, ensuring $\mathbb{E}_{p_0}[u_j] = 0$
and $u_j \in B_{p_0}$. In natural parameters
this gives $\alpha = \theta_1+1$ and $\beta = 1-\theta_2$,
\begin{equation}
\label{eq:gamma_psi}
\psi(\theta) = \log\Gamma(\theta_1+1)
- (\theta_1+1)\log(1-\theta_2)
+ \gamma_E\theta_1 - \theta_2.
\end{equation}

\smallskip\noindent\textbf{Full embedding required.}
The statistics $u_1 = \log x + \gamma_E$ and $u_2 = x - 1$
are algebraically independent since each is a transcendental
function of the other, so condition~\eqref{eq:poly_representation}
of Theorem~\ref{thm:partial_sufficient} cannot hold for any
$J \subsetneq \{1,2\}$. By Remark~\ref{rem:beyond_gaussians},
the full embedding $\mathcal{F}\colon\mathcal{N}\to\mathrm{SPD}(3)$
is required.

The mean parameters and Fisher information matrix, obtained
as first and second partial derivatives of $\psi$ with respect
to $\theta_1, \theta_2$, are
\begin{equation}
\label{eq:gamma_eta}
\eta_1 = \Psi^{(0)}(\alpha) - \log\beta + \gamma_E,
\qquad
\eta_2 = \frac{\alpha}{\beta} - 1,
\end{equation}
\begin{equation}
\label{eq:gamma_fisher}
I(\theta) =
\begin{pmatrix}
\Psi^{(1)}(\alpha) & 1/\beta \\[4pt]
1/\beta & \alpha/\beta^2
\end{pmatrix}.
\end{equation}
Using $A = I + \eta\eta^\top$ from~\eqref{eq:A_I_decomposition},
the embedding matrix is
\begin{equation}
\label{eq:gamma_S}
S(\theta) =
\begin{pmatrix}
\Psi^{(1)}(\alpha) + \eta_1^2 
    & 1/\beta + \eta_1\eta_2 & \eta_1 \\[4pt]
1/\beta + \eta_1\eta_2 
    & \alpha/\beta^2 + \eta_2^2 & \eta_2 \\[4pt]
\eta_1 & \eta_2 & 1
\end{pmatrix} \in \mathrm{SPD}(3).
\end{equation}

\smallskip\noindent\textbf{Metric discrepancy via 
non-vanishing cumulants.}
The third cumulants, obtained as third partial derivatives
of $\psi(\theta)$, are
\begin{equation}
\label{eq:gamma_kappa}
\kappa_{111} = \Psi^{(2)}(\alpha), \qquad
\kappa_{122} = \frac{1}{\beta^2}, \qquad
\kappa_{222} = \frac{2\alpha}{\beta^3},
\end{equation}
with all other independent third cumulants equal to zero.
Since $\kappa_{jlm} \not\equiv 0$, the hypothesis of
Lemma~\ref{lem:kappa3_cancellation} fails and
$ds^2_{\mathcal{M}} \neq ds^2_{\mathrm{Fisher}}$ in general.

\subsection{Geometric Two-Sample Testing}
\label{subsec:hypothesis_testing}

Samples $X=(x_1,\dots,x_m)$ and $Y=(y_1,\dots,y_n)$ are drawn
independently from unknown $q_1, q_2 \in \mathcal{E}_p$. Since both
distributions are non-parametric, the Fisher--Rao distance is
intractable. The framework of Section~\ref{sec:complete_framework}
makes the problem computable. By Theorem~\ref{thm:moment_matching},
each $q_i$ is mapped to its $m$-projection
$p_{\theta_i} = \pi^{(m)}(q_i) \in \mathcal{N}$ via moment matching,
and by Theorem~\ref{thm:spd_embedding} each $p_{\theta_i}$ is
embedded into $\mathrm{SPD}(d+1)$ via $S_i = \mathcal{F}(\theta_i)$.
The pseudo-metric $\tilde{d}(q_1,q_2) = d_{\mathcal{F}}(\theta_1,\theta_2)$
of Definition~\ref{def:projected_distance} is then bounded below in
closed form by Theorem~\ref{thm:complete_inequality}.

The induced null hypothesis is
\begin{equation}
\label{eq:null_hypothesis}
H_0^{\mathcal{N}}\colon \eta^{q_1} = \eta^{q_2}
\qquad\text{against}\qquad
H_1^{\mathcal{N}}\colon \eta^{q_1} \neq \eta^{q_2},
\end{equation}
where $\eta^{q_i} =
(\mathbb{E}_{q_i}[u_1],\dots,\mathbb{E}_{q_i}[u_d])^\top$.
By Proposition~\ref{prop:many_to_one}, $H_0^{\mathcal{N}}$ is
strictly weaker than full distributional equality, since
distributions differing only beyond the moments encoded by
$\mathcal{N}$ satisfy $\tilde{d}(q_1,q_2)=0$.

\smallskip\noindent\textbf{Test statistics.}
Taking $\mathcal{G}$ to be the Gaussian family with $J=\{1\}$
(Section~\ref{subsec:example_gaussian}), the $m$-projections are
identified by their sample moments. The embedding matrices
\begin{equation}
\hat{S}_X = \begin{pmatrix}\hat{A}_X & \hat{\mu}_X \\
\hat{\mu}_X & 1\end{pmatrix}, \qquad
\hat{S}_Y = \begin{pmatrix}\hat{A}_Y & \hat{\mu}_Y \\
\hat{\mu}_Y & 1\end{pmatrix}, \qquad
\hat{\mu} = \frac{1}{m}\sum_{i=1}^m x_i, \quad
\hat{A} = \frac{1}{m}\sum_{i=1}^m x_i^2,
\end{equation}
are computed directly from the samples. The observed AIRM test
statistic is
\begin{equation}
\label{eq:airm_stat}
\hat{T}_{\mathrm{AIRM}}
= d_{\mathrm{AIRM}}(\hat{S}_X, \hat{S}_Y)
= \frac{1}{\sqrt{2}}\bigl\|
  \log(\hat{S}_X^{-1/2}\hat{S}_Y\hat{S}_X^{-1/2})
  \bigr\|_F,
\end{equation}
computable in closed form via~\eqref{eq:airm_distance} and
affine-invariant by~\eqref{eq:airm_congruence}. The two baseline
observed statistics are Hotelling's $T^2$,
\begin{equation}
\label{eq:hotelling}
\hat{T}_H = \frac{mn}{m+n}(\bar{X}-\bar{Y})^\top
S_p^{-1}(\bar{X}-\bar{Y}),
\end{equation}
where $S_p$ is the pooled sample covariance, and the unbiased MMD
with Gaussian kernel $k(x,y)=\exp(-\|x-y\|^2/2\sigma^2)$ and
median bandwidth $\sigma$ (recomputed from the pooled sample at
each permutation step),
\begin{equation}
\label{eq:mmd}
\hat{T}_{\mathrm{MMD}}^2
= \frac{1}{m(m-1)}\sum_{i\neq j}k(x_i,x_j)
- \frac{2}{mn}\sum_{i,j}k(x_i,y_j)
+ \frac{1}{n(n-1)}\sum_{i\neq j}k(y_i,y_j).
\end{equation}
For each method, the observed statistic $\hat{T}_{\mathrm{obs}}$
is calibrated against a null distribution of $B$ permuted
statistics $\hat{T}_1^*,\dots,\hat{T}_B^*$, each computed from
a random relabelling of $Z = X \cup Y$ into groups of sizes $m$
and $n$. The empirical p-value
\begin{equation}
\hat{p} = \frac{1+\#\{b:\hat{T}_b^*\geq\hat{T}_{\mathrm{obs}}\}}{B+1}
\end{equation}
gives exact level $\alpha$ in finite samples under exchangeability.
The complete procedure is given in Algorithm~\ref{alg:perm_test}
and the structural differences between the three methods are
illustrated in Figure~\ref{fig:test_pipeline}.

\begin{algorithm}[H]
\caption{Permutation Two-Sample Test (any statistic $\hat{T}$)}
\label{alg:perm_test}
\begin{algorithmic}[1]
\Require Samples $X=(x_1,\dots,x_m)$, $Y=(y_1,\dots,y_n)$;
         statistic $\hat{T} \in \{\hat{T}_{\mathrm{AIRM}},\,
         \hat{T}_H,\, \hat{T}_{\mathrm{MMD}}\}$;
         permutation count $B$; significance level $\alpha$.
\Ensure Observed statistic $\hat{T}_{\mathrm{obs}}$,
        empirical p-value $\hat{p}$.
\State Compute $\hat{T}_{\mathrm{obs}}$ from $X$ and $Y$.
\State Pool $Z = X \cup Y$.
\For{$b = 1,\dots,B$}
  \State Draw a random permutation $\pi$ of $\{1,\dots,m+n\}$;
         set $X^* = (z_{\pi(1)},\dots,z_{\pi(m)})$,
         $Y^* = (z_{\pi(m+1)},\dots,z_{\pi(m+n)})$.
  \State Compute $\hat{T}_b^*$ from $X^*$ and $Y^*$.
\EndFor
\State $\hat{p} =
  \bigl(1+\#\{b:\hat{T}_b^*\geq\hat{T}_{\mathrm{obs}}\}
  \bigr)/(B+1)$.
\State Reject $H_0^{\mathcal{N}}$ at level $\alpha$ iff
       $\hat{p} < \alpha$.
\end{algorithmic}
\end{algorithm}

\begin{figure}[!htbp]
\centering
\resizebox{\linewidth}{!}{
\begin{tikzpicture}[
  every node/.style={font=\small},
  box/.style={draw, rounded corners=5pt, minimum width=3.8cm,
              minimum height=2.2cm, align=center, thick},
  sbox/.style={draw, rounded corners=5pt, minimum width=8.5cm,
               minimum height=1.1cm, align=center,
               fill=gray!5, thick},
  arr/.style={->, thick, >=stealth, gray!60}
]
\node[sbox] (inp) at (0, 5.2)
  {\textbf{Samples}
   $X=(x_1,\dots,x_m)$ and $Y=(y_1,\dots,y_n)$\\
   from unknown $q_1, q_2\in\mathcal{E}_p$};

\node[box, fill=blue!5, draw=blue!60] (airm) at (-5.2, 2.6) {
  \textbf{AIRM (this paper)}\\[4pt]
  $m$-projection (Thm.~\ref{thm:moment_matching})\\
  $\pi^{(m)}(q_i)\to p_{\theta_i}\in\mathcal{N}$\\[2pt]
  SPD embedding (Thm.~\ref{thm:spd_embedding})\\
  $\hat{S}_i = \mathcal{F}(\hat\theta_i)$\\[4pt]
  $\hat{T}_{\mathrm{obs}} =
    d_{\mathrm{AIRM}}(\hat{S}_X,\hat{S}_Y)$
};
\node[box, fill=gray!5] (hot) at (0, 2.6) {
  \textbf{Hotelling's $T^2$}\\[4pt]
  $\hat{T}_{\mathrm{obs}} = \tfrac{mn}{m+n}
    (\bar{X}-\bar{Y})^\top S_p^{-1}(\bar{X}-\bar{Y})$\\[4pt]
  {\color{red!70!black}\textit{means only}}
};
\node[box, fill=gray!5] (mmd) at (5.5, 2.6) {
  \textbf{MMD}\\[4pt]
  Choose kernel $k$, bandwidth $\sigma$\\[4pt]
  $\hat{T}_{\mathrm{obs}}^2 = \mathbb{E}[k(X,X')]$\\
  $-2\mathbb{E}[k(X,Y)]$\\
  $+\mathbb{E}[k(Y,Y')]$\\[2pt]
  {\color{red!70!black}\textit{bandwidth required}}
};
\draw[arr] (inp.west) -- +(-0.6,0) |- (airm.north);
\draw[arr] (inp.south) -- (hot.north);
\draw[arr] (inp.east) -- +(0.6,0)  |- (mmd.north);
\node[sbox, minimum width=11cm] (perm) at (0, 0) {
  \textbf{Permutation calibration --- identical for all three
  ($\alpha = 0.05$, $B$ permutations)}\\
  Pool $Z=X\cup Y$.\ \
  Shuffle $B$ times; compute $\hat{T}_b^*$ from relabelled data.\ \
  $\hat{p} =
    (1+\#\{b:\hat{T}_b^*\geq\hat{T}_{\mathrm{obs}}\})/(B+1)$.\ \
  Reject $H_0^{\mathcal{N}}$ if $\hat{p}<\alpha$.
};
\draw[arr] (airm.south) -- ++(0,-0.4) -| (perm.north west);
\draw[arr] (hot.south)  -- (perm.north);
\draw[arr] (mmd.south)  -- ++(0,-0.4) -| (perm.north east);
\end{tikzpicture}}
\caption{Pipeline for the three two-sample tests. Each method
computes $\hat{T}_{\mathrm{obs}}$ from the original samples and
$\hat{T}_b^*$ from reshuffled data; the p-value and rejection
rule at $\alpha=0.05$ are identical for all three. AIRM uses
the two-stage construction of
Section~\ref{sec:complete_framework}; Hotelling's $T^2$ compares
means only; MMD requires bandwidth selection.}
\label{fig:test_pipeline}
\end{figure}
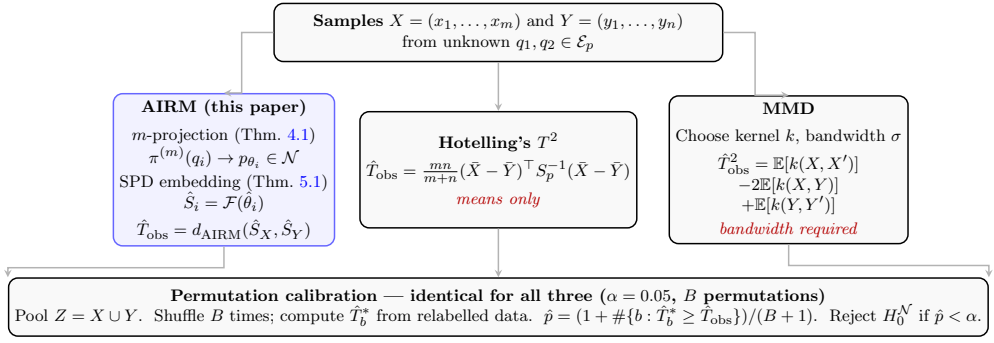

\smallskip\noindent\textbf{Example~1: Non-parametric distributions.}
To demonstrate the proposed framework, samples are drawn from
\begin{equation}
q_1 = \mathrm{Laplace}(0,1), \qquad
q_2 = \mathrm{Logistic}\!\left(0.5,\,\sqrt{4.5}/\pi\right),
\end{equation}
with $\mathbb{E}_{q_1}[x]=0$, $\mathrm{Var}_{q_1}(x)=2$,
$\mathbb{E}_{q_2}[x]=0.5$, $\mathrm{Var}_{q_2}(x)=1.5$.
Both are treated as unknown; the framework uses only the observed
samples $X=(x_1,\dots,x_m)$ and $Y=(y_1,\dots,y_n)$, with
$q_1,q_2\in\mathcal{E}_p$ taken as a modelling assumption on the
reference $p=\mathcal{G}(0,1)$. Taking $\mathcal{G}$ to be
the Gaussian family, Theorem~\ref{thm:moment_matching} reduces
the $m$-projection to estimating mean and raw second moment
from the samples, giving
$\hat{\eta}^{q_i} = (\hat{\mu}_i,\hat{A}_i)^\top$.
At $m=n=200$,
\begin{equation}
\hat{\mu}_X = -0.0737,\quad \hat{A}_X = 1.7758,\qquad
\hat{\mu}_Y = \phantom{-}0.4781,\quad \hat{A}_Y = 1.7159,
\end{equation}
yielding $\hat{\eta}^{q_1} = (-0.0737,\;1.7758)^\top \neq
(0.4781,\;1.7159)^\top = \hat{\eta}^{q_2}$
and SPD embedding matrices
\begin{equation}
\hat{S}_X =
\begin{pmatrix}1.7758 & -0.0737 \\ -0.0737 & 1\end{pmatrix},
\qquad
\hat{S}_Y =
\begin{pmatrix}1.7159 & \phantom{-}0.4781 \\
               0.4781 & 1\end{pmatrix}.
\end{equation}
The observed statistic is
\begin{equation}
\hat{T}_{\mathrm{obs}}
= d_{\mathrm{AIRM}}(\hat{S}_X,\hat{S}_Y)
= \frac{1}{\sqrt{2}}
  \bigl\|\log(\hat{S}_X^{-1/2}\hat{S}_Y
  \hat{S}_X^{-1/2})\bigr\|_F
\approx 0.447.
\end{equation}
With $B=2000$ permutations, the 95th percentile of
$\{\hat{T}_b^*\}_{b=1}^{2000}$ is $0.266$. Since
$\hat{T}_{\mathrm{obs}} = 0.447 > 0.266$,
\begin{equation}
\hat{p} =
\frac{1+\#\{b:\hat{T}_b^*\geq 0.447\}}{2001}
\approx 0.0005 < 0.05 = \alpha,
\end{equation}
and $H_0^{\mathcal{N}}$ is rejected at level $\alpha=0.05$.

\smallskip\noindent\textbf{Example~2: Power under variance shift.}
Both $q_1$ and $q_2$ are $\mathrm{Logistic}(0,\cdot)$ with equal
means, but $\mathrm{Var}(q_1)=1$ fixed and $\mathrm{Var}(q_2)=r$
varying. The scale parameters are
\begin{equation}
s_1 = \frac{\sqrt{3}}{\pi} \approx 0.5513,
\qquad
s_2 = \frac{\sqrt{3r}}{\pi},
\end{equation}
so that $\pi^{(m)}(q_1)=\mathcal{G}(0,1)$ and
$\pi^{(m)}(q_2)=\mathcal{G}(0,r)$. Since means are equal,
$\eta^{q_1}$ and $\eta^{q_2}$ differ only in their second
component, so $H_0^{\mathcal{N}}$ reduces to
$\mathrm{Var}(q_1)=\mathrm{Var}(q_2)$. Hotelling's $T^2$
depends only on $\bar{X}-\bar{Y}\approx 0$ and therefore has
no power against this alternative by construction.

For each value of $r$, the following experiment is repeated
$200$ independent times via Algorithm~\ref{alg:perm_test}. In
each trial, fresh samples of size $m=n=100$ are drawn from
$q_1=\mathrm{Logistic}(0,s_1)$ and
$q_2=\mathrm{Logistic}(0,s_2)$ respectively. The sample mean
$\hat{\mu}$ and raw second moment $\hat{A}$ are computed from
each sample, and the SPD embedding matrices
\begin{equation}
\hat{S}_X = \begin{pmatrix}\hat{A}_X & \hat{\mu}_X \\
\hat{\mu}_X & 1\end{pmatrix},
\qquad
\hat{S}_Y = \begin{pmatrix}\hat{A}_Y & \hat{\mu}_Y \\
\hat{\mu}_Y & 1\end{pmatrix}
\end{equation}
are formed directly from these sample moments. The observed
statistic
\begin{equation}
\hat{T}_{\mathrm{obs}}
= d_{\mathrm{AIRM}}(\hat{S}_X, \hat{S}_Y)
= \frac{1}{\sqrt{2}}\bigl\|
  \log(\hat{S}_X^{-1/2}\hat{S}_Y\hat{S}_X^{-1/2})
  \bigr\|_F
\end{equation}
is computed from these matrices and grows monotonically in $r$
as the variance of $q_2$ increases. All three statistics
$\hat{T}_{\mathrm{AIRM}}$, $\hat{T}_H$, and
$\hat{T}_{\mathrm{MMD}}$ are then calibrated against their
own null distribution of $B=200$ permuted statistics at level
$\alpha=0.05$, and the rejection decision is recorded. After
all $200$ trials, the empirical rejection rate is the fraction
of trials in which $H_0^{\mathcal{N}}$ was rejected.

\begin{table}[!htbp]
\centering
\caption{Empirical rejection rates at $\alpha=0.05$ under pure
variance shift ($q_1=\mathrm{Logistic}(0,s_1)$,
$q_2=\mathrm{Logistic}(0,s_2)$, mean $0$,
$\mathrm{Var}(q_1)=1$ vs.\ $\mathrm{Var}(q_2)=r$,
$s_1=\sqrt{3}/\pi$, $s_2=\sqrt{3r}/\pi$);
$m=n=100$, $200$ trials, $B=200$ permutations.}
\label{tab:power_var}
\renewcommand{\arraystretch}{1.25}
\begin{tabular}{lccc}
\toprule
Variance ratio $r$ & \textbf{AIRM} & MMD & Hotelling $T^2$ \\
\midrule
$1.0$ (size) & $0.07$ & $0.07$ & $0.07$ \\
$1.5$        & $\mathbf{0.31}$ & $0.20$ & $0.07$ \\
$2.0$        & $\mathbf{0.73}$ & $0.54$ & $0.05$ \\
$3.0$        & $\mathbf{0.99}$ & $0.93$ & $0.03$ \\
\bottomrule
\end{tabular}
\end{table}

\noindent AIRM achieves the highest power at every alternative,
with the advantage most pronounced at moderate effect sizes
($r=2$, $0.73$ vs.\ $0.54$ for MMD). The SPD embedding
captures the variance shift directly as a geometric distance
between the two embedding matrices, which explains the power
advantage over MMD. At $r=3$, both AIRM and MMD reach
near-perfect power ($0.99$ and $0.93$ respectively).
Hotelling's $T^2$ remains at the nominal level $\alpha=0.05$
throughout, as expected under a pure variance shift with equal
means; the small fluctuations between $0.03$ and $0.07$
across rows arise from running the experiment on different
samples each trial.

\smallskip\noindent\textbf{Limitation.}
By Proposition~\ref{prop:many_to_one}, $\tilde{d}$ is blind to
distributional features beyond the moments encoded by $\mathcal{N}$.
Both $q_1=\mathrm{Laplace}(0,1/\sqrt{2})$ and
$q_2=\mathrm{Logistic}(0,\sqrt{3}/\pi)$ lie in $\mathcal{E}_{p_0}$
and satisfy
\begin{equation}
\mathbb{E}_{q_1}[x] = \mathbb{E}_{q_2}[x] = 0,
\qquad
\mathrm{Var}_{q_1}(x) = \mathrm{Var}_{q_2}(x) = 1,
\end{equation}
so their $m$-projections onto the Gaussian family coincide,
$\pi^{(m)}(q_1)=\pi^{(m)}(q_2)=\mathcal{G}(0,1)$,
and consequently $\tilde{d}(q_1,q_2)=0$.
The test is therefore unable to distinguish $q_1$ from $q_2$,
since $\tilde{d}$ assigns them zero distance.
Adding higher-order sufficient statistics to $\mathcal{N}$
incorporates moment information beyond mean and variance,
grows the embedding to a larger SPD manifold, and restores
power against such alternatives with no other change to
Algorithm~\ref{alg:perm_test}.

\section*{Acknowledgments}
Amit Vishwakarma is thankful to the Indian Institute of Space Science and Technology, Department of Space, Government of India for the award of the doctoral research fellowship.

\section*{Funding}
No funding was received for this work.

\section*{Conflict of Interests}
The authors declare no conflict of interest.

\end{document}